\documentclass[reqno,11pt]{amsart}

\usepackage[top=2.0cm,bottom=2.0cm,left=3cm,right=3cm]{geometry}
\usepackage{amsthm,amsmath,amssymb,amsfonts,dsfont,mathrsfs}
\usepackage{functan,extarrows,mathtools}
\usepackage{xcolor}
\usepackage[colorlinks,linkcolor=black,citecolor=black]{hyperref}
\usepackage{cite}
\usepackage{indentfirst,latexsym,enumerate,graphicx}
\usepackage{float,relsize,marginnote,stmaryrd,esint,bm}
\usepackage{caption,subfigure,paralist,appendix}

\allowdisplaybreaks

\theoremstyle{plain}
\newtheorem{thm}{Theorem}[section]
\newtheorem{cor}[thm]{Corollary}
\newtheorem{lem}[thm]{Lemma}

\newtheorem{rem}{Remark}[section]

\numberwithin{equation}{section}

\DeclareMathOperator{\dive}{div}

\begin{document}

\title[Compressible Navier--Stokes equations with a potential force]
{Compressible Navier--Stokes equations with a potential force:
global well-posedness and optimal time-decay rates for arbitrarily large $L^2$ initial data}

\author[J. Ni]{Jinkai Ni}
\address[JKN]{School of Mathematics, Nanjing University, Nanjing
 210093, P. R. China}
\email{jinkaini123@gmail.com}

\author[L. Wang] {Luqi Wang $^*$} \thanks{$^*$\! Corresponding author}
\address[LQW]{School of Mathematics, Nanjing University, Nanjing
 210093, P. R. China}
\email{wangluqi@nju.edu.cn}

\author[Z. Zhang]{Zhipeng Zhang}
\address[ZPZ]{School of Mathematical Sciences,
Ocean University of China, Qingdao 266100, P. R. China}
\email{zhangzp@ouc.edu.cn}

\begin{abstract}
We study the Cauchy problem for the three-dimensional barotropic
compressible Navier--Stokes equations with a time-independent
potential force near a spatially nonconstant stationary state.
The potential is controlled in unweighted homogeneous Besov spaces; in particular, no polynomial spatial-weight condition involving $(1+|x|)^j\nabla^j\phi$ is imposed.
For initial data relative to the stationary state that are sufficiently small in $\dot H^{\frac12-\delta}\cap\dot H^3$, we establish the existence and uniqueness of a global strong solution in $H^3$, while allowing the initial $L^2$ norm to be arbitrarily large. If the initial data are bounded in $\dot B^s_{2,\infty}$ for
$s\in[-\frac32,-1)$, then the solution and its first spatial
derivative decay at the optimal rates $(1+t)^{-\frac{k-s}{2}}$ with
$k=0$ and $1$, respectively. The analysis relies on refined homogeneous energy estimates and a frequency-localized description for the dissipative and asymptotic structures of the system.
\end{abstract}



\maketitle

\setcounter{equation}{0}

\section{Introduction and main results}

\subsection{Background and motivation}
We consider the Cauchy problem for the following compressible Navier--Stokes (CNS) equations with a time-independent potential force in the whole
space $\mathbb R^3$:
\begin{equation}\label{I1}\left\{
\begin{aligned}
& \partial_t \rho+{\rm div}\,(\rho u)=0, \\
&  \partial_t u+u\cdot\nabla u+\frac{\nabla P(\rho)}{\rho}-\frac{\mu}{\rho}\Delta u-\frac{\mu+\lambda}{\rho}\nabla{\rm div}\,u=-\nabla\phi(x), \\
\end{aligned}\right.
\end{equation}
supplemented with the initial data
\begin{align}\label{I2}
(\rho,u)(0,x)=(\rho_0,u_0)(x)\rightarrow(\rho_{\infty},0),
\qquad |x|\rightarrow \infty.
\end{align}
Here, $t>0$ and $x\in\mathbb R^3$. The unknowns $\rho=\rho(t,x)>0$ and $u=u(t,x)\in\mathbb R^3$ denote the density and velocity of the fluid, respectively. The pressure function $P=P(\rho)$ is smooth near the far-field density
$\rho_\infty>0$ and satisfies $P'(\rho_\infty)>0$. 
The viscosity
coefficients $\mu$ and $\lambda$ obey
$\mu>0$ and $3\lambda+2\mu\geq0$.
The term $-\nabla\phi(x)$ represents the time-independent potential force. 
The associated
stationary state is $(\rho_*(x),0)$, where $\rho_*(x)$ satisfies
\begin{align}\label{A1}
\int_{\rho_{\infty}}^{\rho_{*}(x)}
 \frac{P'(z)}{z}\,{\rm d}z+\phi(x)=0,
\end{align}
see, for example, \cite{MN-CMP-1983}.

The well-posedness theory for compressible viscous flows has a long
history. Global classical solutions near a constant non-vacuum state
were first established by Matsumura and Nishida
\cite{MN-PJAA-1979,MN-JMKU-1980}. Local strong solutions in Sobolev
spaces and global solutions allowing large oscillations were developed
in \cite{CK-JDE-2003,CCK-JMPA-2004,HLX-CPAM-2012,HHW-ARMA-2019}.
Critical Besov methods provide another natural framework for separating
the genuinely nonlinear scales from the low-frequency part of the
data, see \cite{BCD-Book-2011,DX-ARMA-2017,XX-JDE-2021} and the
references therein.

For a constant equilibrium, the large-time behavior is governed by the
diffusion-wave structure of the linearized hyperbolic-parabolic
system. Optimal decay estimates were obtained by Ponce
\cite{Ponce-NA-1985}, Hoff and Zumbrun \cite{HZ-IUMJ-1995}, and Liu
and Wang \cite{LW-CMP-1998}. Guo and Wang \cite{GW-CPDE-2012}
developed a pure energy method based on negative Sobolev norms. In the critical regularity setting, 
Danchin and Xu \cite{DX-ARMA-2017} and
Xin and Xu \cite{XX-JDE-2021} used low-high frequency analysis and
negative Besov norms to weaken the low-frequency smallness imposed in earlier approaches. 
Related advances for rough or highly oscillatory
data and for refined high-order decay can be found in
\cite{HW-JDE-2020,LY-JDE-2025}.

The presence of a potential force changes the problem substantially,
because the equilibrium density is no longer spatially homogeneous.
Stability of nonconstant stationary states was studied by Matsumura and
Nishida \cite{MN-CMP-1983} and Matsumura and Padula
\cite{MP-SAACM-1992}. Explicit convergence estimates were obtained by
Deckelnick \cite{Dk-MZ-1992}, Shibata and Tanaka
\cite{ST-JMSJ-2003,ST-CMA-2007}, and Ukai, Yang and Zhao
\cite{UYZ-JHDE-2006}. Duan, Liu, Ukai and Yang
\cite{DLUY-JDE-2007} combined nonlinear energy estimates with
$L^p$--$L^q$ estimates on the semigroup generated by the corresponding linearized operator.
Under an additional
$L^p$ assumption with $1\leq p<\frac65$, they obtained the optimal decay rates of the solution and its first derivative. 
Further results on semigroup and energy decay were given in
\cite{DUYZ-MMMAS-2007,Okita-KJM-2014,Wang-NARWA-2017}.

Gao, Li and Yao \cite{GLY-2023-JDE} subsequently established the
upper and lower bounds of decay rates for higher spatial derivatives by combining
energy estimates, spectral analysis, and a low-high frequency
decomposition. A key assumption in their treatment of the
spatially dependent coefficients is the weighted localization
condition
\begin{align}\label{G1.1}
 \sum_{j=0}^{N+1}
 \big\|(1+|x|)^j\nabla^j\phi\big\|_{L^2\cap L^\infty}
 \ll1.
\end{align}
The weights permit a derivative transfer through Hardy-type estimates,
but they impose strong localization on the potential. Lower bounds for
dissipative systems originate in the Fourier-splitting work of
Schonbek \cite{Schonbek-JAMS-1991}; higher-order and decay-character
methods were developed in \cite{OT-JFA-2000,Bl-2016-SIMA}. More
recently, frequency-localized Besov techniques have clarified how the
linear low-frequency profile determines sharp decay, see
\cite{DX-ARMA-2017,XX-JDE-2021,IOH-JDE-2025}.

A closely related result is due to Deguchi
\cite[Theorems~1.2 and 1.5, and Remark~1.4]{Deguchi-MA-2024}.
He assumed the external force $ F\in\dot B^{-\frac32}_{2,\infty}\cap\dot H^3$
and small initial data in
$\dot B^{\frac12}_{2,\infty}\cap\dot H^3$, without any $L^2$
assumption. When $F=-\nabla\phi$, his force condition is equivalent to $\phi\in\dot B^{-\frac12}_{2,\infty}\cap\dot H^4$.
Our initial-data space satisfies the embedding chain
\begin{align*}
 \dot H^{\frac12-\delta}\cap\dot H^3
 \hookrightarrow
 \dot H^{\frac12}
 \hookrightarrow
 \dot B^{\frac12}_{2,\infty}.
\end{align*}
Thus, our initial-data assumption is stronger, and Deguchi's result
already yields the upper decay rates below whenever his force
assumption is satisfied.
The main difference lies in the admissible potentials. We assume
\begin{align*}
 \phi\in
 \dot B^{\frac12}_{2,1}\cap\dot B^{\frac92}_{2,1},
\end{align*}
which requires higher regularity at high frequencies but no
negative-order localization at low frequencies, and is therefore
incomparable with
$\dot B^{-\frac12}_{2,\infty}\cap\dot H^4$. Indeed, let
$\psi\in\mathcal S(\mathbb R^3)$ have Fourier support in a fixed
annulus, and set
\begin{align*}
 \psi_j(x):=2^{\frac32j}\psi(2^j x),
 \qquad
 \phi_\sharp
 :=
 \epsilon\sum_{j\leq-1}|j|2^{\frac j2}\psi_j .
\end{align*}
Then, up to finite dyadic overlaps,
\begin{align*}
 \|\phi_\sharp\|_{\dot B^{\frac12}_{2,1}}
 +\|\phi_\sharp\|_{\dot B^{\frac92}_{2,1}}
 &\lesssim
 \epsilon\sum_{j\leq-1}|j|
 \bigl(2^j+2^{5j}\bigr)<\infty,\\
 \|\phi_\sharp\|_{\dot B^{-\frac12}_{2,\infty}}
 &\gtrsim
 \epsilon\sup_{j\leq-1}|j|=\infty.
\end{align*}
Hence, $\phi_\sharp$ is admissible in the present framework for
sufficiently small $\epsilon$, but is excluded by Deguchi's
low-frequency condition. The additional assumption
$\phi\in\dot B^{s+\frac32}_{2,\infty}$ in the decay result is likewise
weaker at low frequencies.

For this unweighted class of potentials, we prove global
well-posedness by homogeneous energy estimates and allow the $L^2$-norm of the initial data to be large through spatial spreading. We also
obtain the zeroth- and first-order decay rates under a finite, not
necessarily small, norm assumption in $\dot B^s_{2,\infty}$. Since
\begin{align*}
 L^{p_s}(\mathbb R^3)\hookrightarrow\dot B^s_{2,\infty}(\mathbb R^3),
 \qquad p_s:=\frac6{3-2s}\in\left[1,\frac65\right),
\end{align*}
this assumption is weaker than the $L^{p_s}$ condition in
\cite{DLUY-JDE-2007}. Finally, an explicit decay-character condition
on the prescribed data yields lower bounds matching the upper rates, which strengthens the special-data optimality result in
\cite[Theorem~1.6]{Deguchi-MA-2024}.

\subsection{Main results}
Before stating our main results, we first reformulate the original system \eqref{I1}--\eqref{I2} using the following perturbation variables:
\begin{align*}
\varrho(t,x)=\rho(t,x)-\rho_*(x), \quad
\omega(t,x)=\frac{\rho_{\infty}}{\sqrt{P'(\rho_{\infty})}}u(t,x),\quad \bar\rho(x)=\rho_{*}(x)-\rho_{\infty}.
\end{align*}
With these definitions, the system \eqref{I1}--\eqref{I2} can be rewritten as
\begin{equation}\label{I3}\left\{
\begin{aligned}
& \partial_t \varrho+ \gamma{\rm div}\,\omega=\mathcal{N}_1, \\
&  \partial_t  \omega-\mu_1\Delta\omega-\mu_2\nabla{\rm div}\,\omega+\gamma\nabla\varrho=\mathcal{N}_2, \\
&(\varrho,\omega)(0,x)=(\varrho_0,\omega_0)(x)=\bigg(\rho_0-\rho_*,\frac{\rho_{\infty}}{\sqrt{P'(\rho_{\infty})}}u_0\bigg)(x)\rightarrow(0,0), \quad |x|\rightarrow \infty,
\end{aligned}\right.
\end{equation}
where
\begin{align*}
 \mu_1=\frac{\mu}{\rho_\infty},\qquad
 \mu_2=\frac{\mu+\lambda}{\rho_\infty},\qquad
 \gamma=\sqrt{P'(\rho_\infty)},
\end{align*}
and
\begin{equation}\label{I4}\left\{
\begin{aligned}
  \mathcal{N}_1=&\,-\frac{\mu_1\gamma}{\mu}{\rm div}\, [(\varrho+\bar\rho)\omega],\\
\mathcal{N}_2=&\,-\frac{\mu_1\gamma}{\mu}\omega\cdot\nabla\omega-\mu_1\frac{\varrho+\bar\rho}{\varrho+\rho_*}\Delta\omega-\mu_2\frac{\varrho+\bar\rho}{\varrho+\rho_*}\nabla{\rm div}\,\omega \\
&\, -\frac{\mu}{\mu_1\gamma}\bigg[ \frac{P'(\varrho+\rho_{*})}{\varrho+\rho_*}-\frac{P'(\rho_*)}{\rho_*} \bigg]\nabla\bar\rho-\frac{\mu}{\mu_1\gamma}\bigg[ \frac{P'(\varrho+\rho_*)}{\varrho+\rho_*}-\frac{P'(\rho_{\infty})}{\rho_{\infty}}  \bigg]\nabla\varrho.
\end{aligned}\right.
\end{equation}

We now state the main results.
\begin{thm}\label{T1.1}
Let $(\rho_*,0)$ be the stationary solution associated with
\eqref{I1}. Assume that $(\rho_0-\rho_*,u_0)\in L^2$ and that the
potential function $\phi$ and the initial perturbation satisfy
\begin{align}
\label{G1.7}
\|\phi\|_{\dot B^{\frac{1}{2}}_{2,1}\cap \dot B^{\frac{9}{2}}_{2,1}}+\|(\rho_0 - \rho_*,u_0)\|_{\dot H^{\frac{1}{2}-\delta}\cap \dot H^3}&\leq \varepsilon_0,
\end{align}
for some sufficiently small constant $\varepsilon_0>0$, where
$\delta\in(0,\frac12)$ is fixed but arbitrarily small. Then the
Cauchy problem \eqref{I3} admits a unique global strong solution with
\begin{align*}
 &\varrho\in C([0,\infty);H^3)\cap C^1([0,\infty);H^2),\\
 &\omega\in C([0,\infty);H^3)\cap C^1([0,\infty);H^1),
\end{align*}
satisfying
\begin{align}\label{G1.8}
\|(\varrho,\omega)(t)\|_{H^3}^2+\int_0^t \big(\|\nabla\varrho(\tau)\|_{H^2}^2+\|\nabla\omega(\tau)\|_{H^3}^2\big){\rm d}\tau \leq C \|(\varrho_0,\omega_0)\|_{H^3}^2,
\end{align}
for all $t\geq0$, where $C>0$ is independent of $t$.
\end{thm}

\begin{rem}
Theorem~\ref{T1.1} imposes smallness only on positive-order
homogeneous norms. Hence, $\|(\rho_0-\rho_*,u_0)\|_{L^2}$ may be
arbitrarily large. It also avoids the polynomially weighted
assumption \eqref{G1.1} from \cite{GLY-2023-JDE} and the weighted
localization imposed in \cite{DLUY-JDE-2007}. Besides, the embedding
$\dot B^{\frac32}_{2,1}\hookrightarrow L^\infty$ gives
\begin{align*}
 \|(\rho_0-\rho_*,u_0)\|_{L^\infty}\lesssim\varepsilon_0.
\end{align*}
Thus,``large $L^2$ initial data'' means small-amplitude data with large
spatial extent, not data with large pointwise oscillation.
\end{rem}

\begin{thm}\label{T1.2}
Let $\varepsilon_0$ and $\delta$ be as in
Theorem~\ref{T1.1}, and let
$s\in[-\frac32,-1)$. There exists a constant
$\varepsilon_1\in(0,\varepsilon_0)$, independent of
$\|(\rho_0-\rho_*,u_0)\|_{L^2}$, such that the following statement
holds. Assume that $(\rho_0-\rho_*,u_0)\in L^2$,
\begin{align}\label{G1.10}
 \|(\rho_0-\rho_*,u_0)\|_{\dot B_{2,\infty}^s}<+\infty,
\end{align}
and for any $0<\varepsilon\leq\varepsilon_1$,
\begin{align}\label{G1.9}
 &\|\phi\|_{
 \dot B^{s+\frac32}_{2,\infty}
 \cap\dot B^{\frac92}_{2,1}}
 +\|(\rho_0-\rho_*,u_0)\|_{
 \dot H^{\frac12-\delta}\cap\dot H^3}
 \leq\varepsilon.
\end{align}
Then the solution $(\varrho,\omega)$ given by
Theorem~\ref{T1.1} satisfies
\begin{align*}
 \sup_{t\geq0}
 \|(\varrho,\omega)(t)\|_{\dot B^s_{2,\infty}}
 \leq
 C\|(\varrho_0,\omega_0)\|_{\dot B^s_{2,\infty}},
\end{align*}
and
\begin{align*}
 \|\nabla^k(\rho-\rho_*)(t)\|_{L^2}
 +\|\nabla^k\omega(t)\|_{L^2}
 \leq
 C(1+t)^{-\frac{k-s}{2}},
 \qquad k=0,1,
\end{align*}
where $C>0$ is independent of $t$.
\end{thm}

\begin{rem}
For $s<-1$, a dyadic decomposition at frequency one yields
\begin{align*}
 \|\phi\|_{\dot B^{\frac12}_{2,1}}
 \lesssim_s
 \|\phi\|_{\dot B^{s+\frac32}_{2,\infty}}
 +\|\phi\|_{\dot B^{\frac92}_{2,1}}.
\end{align*}
Consequently, after decreasing $\varepsilon_1$ if necessary,
\eqref{G1.9} implies
\begin{align*}
 \|\phi\|_{
 \dot B^{\frac12}_{2,1}
 \cap\dot B^{\frac92}_{2,1}}
 \leq\varepsilon_0.
\end{align*}
Thus the potential assumption of Theorem~\ref{T1.1} is automatically
satisfied.
\end{rem}

\begin{rem}\label{R1.2}
The assumption \eqref{G1.10} is weaker than the additional
$L^{p_s}$ condition imposed in \cite{DLUY-JDE-2007}. In particular,
the initial $\dot B^s_{2,\infty}$ norm is only required to be finite
and need not be small; this is consistent with the low-frequency
framework developed in \cite{Xu-CMP-2019}.
The restriction to $k=0,1$ is due to the spatial dependence of the
stationary density. Indeed, starting from the second derivative level,
differentiating terms such as
$\operatorname{div}(\bar\rho\,\omega)$ produces
$(\nabla^3\bar\rho)\omega$. Although the stationary profile is small,
this term contains no derivative of $\omega$ and can only be controlled
through the first-order dissipation. It therefore cannot be absorbed
in a second-order decay estimate. In \cite{GLY-2023-JDE}, this
derivative loss is overcome by weighted Hardy estimates under stronger
spatial localization assumptions. Such a mechanism is not available
in the present unweighted setting.
\end{rem}

Observe that the embeddings: $ L^1(\mathbb R^3) \hookrightarrow
\dot B^{-\frac32}_{2,\infty}(\mathbb R^3)$, $L^{\frac65}(\mathbb R^3) \hookrightarrow
\dot H^{-1}(\mathbb R^3)\hookrightarrow
\dot B^{-1}_{2,\infty}(\mathbb R^3)$ and $ H^{\frac92+\delta}(\mathbb R^3) \hookrightarrow
\dot B^{\frac92}_{2,1}(\mathbb R^3)$ hold for any $\delta>0$. Then the endpoint
$s=-\frac32$ in Theorem \ref{T1.2} gives the following consequence.

\begin{cor}\label{cor1.3}
Let $\varepsilon_0$ be as in Theorem~\ref{T1.1}, and assume that
$(\rho_0-\rho_*,u_0)\in L^2$. Then there exists a constant
$\varepsilon_2\in(0,\varepsilon_0)$, independent of the size of the
initial $L^2$ norm, such that, for any
$\varepsilon\leq\varepsilon_2$, if
\begin{align*}
\|\phi\|_{H^{\frac{9}{2}+\delta} }+\|(\rho_0 - \rho_*,u_0)\|_{\dot H^{\frac{1}{2}-\delta}\cap \dot H^3}&\leq \varepsilon,
\end{align*}
for some fixed but arbitrarily small
$\delta\in(0,\frac12)$, and if
\begin{align*}
\|(\rho_0-\rho_*,u_0)\|_{L^1}<+\infty,
\end{align*}
then the solution $(\varrho,\omega)$ given by
Theorem~\ref{T1.1} satisfies
\begin{align*}
\|\nabla^k(\rho-\rho_*)(t)\|_{L^2}+\|\nabla^k\omega(t)\|_{L^2}\leq&\, C(1+t)^{-\frac{2k+3}{4}},
\end{align*}
for $k=0,1$, where $C$ is independent of $t$.
\end{cor}

To formulate the lower bounds, we use the decay-character class
introduced in \cite[Section~3]{Bl-2016-SIMA}. For
$s_1\in\mathbb R$, define
\begin{equation}\label{G1.15}
\dot{\mathfrak B}_{2,\infty}^{s_1}:=
\left\{ f\in\dot B_{2,\infty}^{s_1}\,\Bigg|
 \begin{array}{l}
 \exists\,\mathfrak a_*,L_*>0,\quad
 \exists\,\{j_m\}_{m\geq1}\subset\mathbb Z
 \text{ strictly decreasing},\\
 j_m\to-\infty,\quad 1\leq j_m-j_{m+1}\leq L_*,\\
 2^{s_1j_m}\|\dot\Delta_{j_m}f\|_{L^2}\geq\mathfrak a_*
 \quad\text{for any }m\geq1
 \end{array}
 \right\}.
\end{equation}
For a vector-valued function $f$, the $L^2$ norm in \eqref{G1.15} is understood
as the product norm of its components.

\begin{thm}\label{T1.4}
Under the assumptions of Theorem~\ref{T1.2}, assume in addition that
$(\rho_0-\rho_*,u_0)$ satisfies
\begin{align}\label{G1.16}
(\rho_0 - \rho_*, u_0) \in \dot{\mathfrak B}_{2,\infty}^{s}
\quad\text{with}\quad s \in \left[-\frac{3}{2}, -1\right).
\end{align}
Then, for the global strong solution $(\varrho,\omega)$ to
\eqref{I3}, there exists $T_1>0$ such that, for all $t\geq T_1$,
\begin{align}\label{G1.17}
c(1 + t)^{-\frac{k - s}{2}} \leq \|\nabla^k(\rho - \rho_*,\omega)(t) \|_{L^2} \leq C(1 + t)^{-\frac{k - s}{2}},
\end{align}
for $k=0,1$, where $c,C>0$ are independent of $t$.
\end{thm}

\begin{rem}
The decay rate in Theorem~\ref{T1.4} is optimal because the upper and
lower bounds have exactly the same algebraic exponent. The lower bound
is stated for the coupled density-velocity norm, while separate componentwise
lower bounds would require additional non-cancellation conditions on
the acoustic modes.
\end{rem}

At the endpoint $s=-\frac32$, the $L^1$ condition provides a simple
criterion for membership in the decay-character class.

\begin{cor}\label{cor1.5}
Under the assumptions of Corollary~\ref{cor1.3}, assume in addition
that at least one of
$\int_{\mathbb R^3}(\rho_0-\rho_*)(x){\rm d}x$ and
$\int_{\mathbb R^3}u_0(x){\rm d}x$ is nonzero. Then the global
solution $(\varrho,\omega)$ obtained in
Corollary~\ref{cor1.3} satisfies, for all sufficiently large $t$,
\begin{align*}
c(1 + t)^{-\frac{2k + 3}{4}} \leq \|\nabla^k(\rho - \rho_*,\omega)(t) \|_{L^2} \leq C(1 + t)^{-\frac{2k + 3}{4}},
\end{align*}
for $k=0,1$, where $c,C>0$ are independent of $t$.
\end{cor}

\begin{rem}
The endpoint decay rates in Corollary~\ref{cor1.5} are comparable with the
$k=0,1$ part of Gao, Li and Yao \cite{GLY-2023-JDE}, but neither set
of hypotheses contains the other. Their result is formulated in a
negative Sobolev framework, assumes polynomially weighted control of
the potential, and gives upper and lower bounds for all derivative orders.
Here the potential assumptions are unweighted, the lower bound is for
the coupled norm, and only the orders $k=0,1$ are asserted.
\end{rem}

\subsection{Strategy of the proofs}
We outline the main mechanisms and the points at which the unweighted
assumptions enter.

\medskip
\noindent\textit{Step 1: global existence with an unrestricted $L^2$
component.}
We combine the endpoint estimates in
$\dot H^{\frac12-\delta}$ and $\dot H^3$ with the full $H^3$ energy.
The bootstrap bounds yield
\begin{align*}
 \frac{{\rm d}}{{\rm d}t}\mathcal E_3(t)
 +\lambda_3\mathcal D_3(t)\leq0.
\end{align*}
The zero-order $L^2$ component remains in $\mathcal E_3$ and never
appears as a small coefficient.

\medskip
\noindent\textit{Step 2: optimal upper decay bounds.}
The Shizuta--Kawashima structure
\cite{KMN-CMP-1979,SK-HMJ-1985,XK-ARMA-2015} gives the multiplier
$|\xi|^2/(1+|\xi|^2)$, hence heat-like decay at low frequencies and
exponential damping at high frequencies. Localized estimates propagate
the negative Besov norm, interpolation yields the zeroth-order rate,
and a low-frequency Duhamel estimate together with a high-frequency
Lyapunov inequality gives the first-order rate. The obstruction in
Remark~\ref{R1.2} limits the argument to $k=0,1$.

\medskip
\noindent\textit{Step 3: matching lower bounds.}
A reverse localized energy inequality preserves the initial size of a
sufficiently low dyadic block up to errors that vanish as
$j\to-\infty$. The decay-character condition provides such blocks
with bounded gaps, and the choice $2^{2j}(1+t)\sim1$ yields the
lower rate matching the upper one, see
\cite{Schonbek-JAMS-1991,Bl-2016-SIMA}.

\subsection{Organization of the paper}
The rest of the paper is organized as follows.
Section~\ref{S2} introduces the notation, the low-high frequency
decomposition, and the Sobolev--Besov product estimates used
throughout this paper. Section~\ref{S3} establishes the homogeneous endpoint
bounds and the global $H^3$ energy estimate, thereby proving
Theorem~\ref{T1.1}. Section~\ref{S4} develops the linear spectral estimate,
propagates the negative Besov norm, and combines low-frequency
Duhamel estimates with a high-frequency Lyapunov inequality to prove
Theorem~\ref{T1.2}. Section~\ref{S5} derives the localized dyadic
energy inequalities and proves Theorem~\ref{T1.4} and
Corollary~\ref{cor1.5}.

\section{Preliminaries}\label{S2}

\subsection{Notation}
Throughout this paper, $A\lesssim_{\alpha} B$ means that there exists a positive constant $C$, depending only on $\alpha$, such that $A\leq CB$. We write $A\sim_{\alpha} B$ when both $A\lesssim_{\alpha} B$ and $B\lesssim_{\alpha} A$ hold. For two operators $A$ and $B$, their commutator is defined by $[A,B]=AB - BA$. We denote by $\mathcal{S}$ the Schwartz class on $\mathbb{R}^3$ and by $\mathcal{S}^{\prime}$ the space of tempered distributions on $\mathbb{R}^3$. For a Banach space $X$, we use the shorthand $\|(a,b)\|_{X}:=\|a\|_{X}+\|b\|_{X}$ for $a,b\in X$.

We use $\langle\cdot, \cdot\rangle$ to denote the standard $L^2$ inner product on $\mathbb{R}^3$, namely, \begin{align*}\langle f,g \rangle=\int_{\mathbb{R}^3} f(x)g(x) \mathrm{d}x,\end{align*} for $f,g\in L^2(\mathbb{R}^3)$. We use $(f|g):=f\cdot \bar g$ to denote the dot product of $f$ with the complex conjugate of $g$.
For an integrable function $f:\mathbb{R}^3\rightarrow \mathbb{R}$, we define its Fourier transform by
\begin{align*}
\mathcal{F}(f)(\xi)=\hat{f}(\xi)=\int_{\mathbb{R}^{3}}e^{-ix\cdot\xi}f(x)\mathrm{d}x,
\end{align*}
where $i = \sqrt{-1}\in\mathbb{C}$ and $x\cdot\xi=\sum_{j = 1}^{3}x_{j}\xi_{j}$ for any $\xi\in\mathbb{R}^{3}$.

Let $s\in \mathbb{R}$. The homogeneous Sobolev space $\dot H^s=\dot H^s(\mathbb{R}^3)$ consists of tempered distributions $f$ on $\mathbb{R}^3$ such that $\hat{f}\in L^1_{\rm loc}$ and $\|f\|_{\dot H^s}:=\||\cdot|^s\hat{f}\|_{L^2}<\infty$. The inhomogeneous Sobolev space $H^s = H^s(\mathbb{R}^3)$ consists of tempered distributions $f$ on $\mathbb{R}^3$ such that $\|f\|_{H^s}:=\|(1 + |\cdot|)^s \hat{f}\|_{L^2}<\infty$. For $1\leq p \leq \infty$, the norm on $L^p(\mathbb{R}^3)$ is denoted by $\|\cdot\|_{L^p}$. 
For multi-indices $\alpha=(\alpha_1,\alpha_2,\alpha_3)$ and $\beta=(\beta_1,\beta_2,\beta_3)$, we write \begin{align*}\partial_x^\alpha=\partial^{\alpha_1}_{x_1}\partial^{\alpha_2}_{x_2} \partial^{\alpha_3}_{x_3},\quad |\alpha|=\sum_{j = 1}^3\alpha_j,
\end{align*} 
and  $\beta\leq \alpha$ if $\beta_j\leq \alpha_j$ for $j=1,2,3$.
Moreover, $C_{\beta}^\alpha=\binom{\alpha}{\beta}=\Pi_{j=1}^3\binom{\alpha_j}{\beta_j}$ when $\beta\leq \alpha$.

We next decompose functions into low- and high-frequency components. For $f\in L^2(\mathbb{R}^3)$, define
\begin{align}\label{G2.1}
f^{L}(x) := \Phi_0(D_x)f(x), \quad f^{H}(x) := \Phi_1(D_x)f(x),
\end{align}
where $\Phi_0(D_x)$ and $\Phi_1(D_x)$ are pseudodifferential operators with symbols $\Phi_0(\xi)$ and $\Phi_1(\xi)$, respectively. Here $D_x = \frac{1}{\sqrt{-1}}(\partial_{x_1}, \partial_{x_2}, \partial_{x_3})$, and $\Phi_0$ and $\Phi_1$ are smooth cutoff functions, with $\Phi_0$ satisfying \begin{equation}\label{G2.2}
\Phi_0(\xi) =
\begin{cases}
1, & |\xi| < \frac{r_0}{2}, \\
0, & |\xi| > r_0,
\end{cases}
\end{equation}
for some constant $r_0>0$, and $\Phi_1(\xi):=1-\Phi_0(\xi)$.
It follows from \eqref{G2.1} and \eqref{G2.2} that \begin{align*}
f(x) = f ^{L}(x) + f^{H}(x).
\end{align*}
Moreover, for any $f \in H^3(\mathbb{R}^3)$, Plancherel's theorem yields the following Bernstein-type inequalities, namely for $0<m\leq 3$,
\begin{align} \label{G2.3}
\|f^H\|_{L^2 } \lesssim \|\nabla^m f\|_{L^2},
\end{align}
and for $0\leq n<3$,
\begin{align*}
\|\nabla^3f^L\|_{L^2} \lesssim&\,  \|\nabla^n f^L\|_{L^2 }.
\end{align*}

Finally, we recall the definition of homogeneous Besov spaces. Let $\varphi\in C_0^{\infty}(\mathbb{R}^3)$ be a smooth function such that $\varphi(\xi)=1$ when $|\xi|\leq 1$ and $\varphi(\xi)=0$ when $|\xi|\geq 2$. Define $\varsigma(\xi)=\varphi(\xi)-\varphi(2\xi)$ and $\varsigma_k(\xi)=\varsigma(2^{-k}\xi)$ for $k\in \mathbb{Z}$. By construction, we have $\sum\limits_{k\in\mathbb{Z}}\varsigma_k(\xi)=1$ for all $\xi\neq 0$. Set $\dot{\Delta}_k f:=\mathcal{F}^{-1}(\varsigma_k)*f$, where $\mathcal{F}^{-1}$ denotes the inverse Fourier transform, and define
\begin{align*}
 \dot S_k f:=\sum_{\ell<k}\dot\Delta_\ell f,
 \qquad
 \widetilde{\dot\Delta}_k f
 :=\dot\Delta_{k-1}f+\dot\Delta_kf+\dot\Delta_{k+1}f.
\end{align*}
We use the standard homogeneous realization
\begin{align*}
 \mathcal S'_h(\mathbb R^3)
 :=\left\{f\in\mathcal S'(\mathbb R^3):
 \lim_{k\to-\infty}\dot S_kf=0\ \text{in }\mathcal S'\right\}.
\end{align*}
The Littlewood--Paley decomposition is then given by
$$f = \sum_{k\in\mathbb{Z}}\dot{\Delta}_{k}f$$
in $\mathcal S'_h(\mathbb R^3)$. For $s\in\mathbb{R}$ and $1\leq p,r\leq \infty$, the homogeneous Besov space $\dot{B}^s_{p,r}(\mathbb{R}^3)$ is defined by
$$\dot B^{s}_{p,r}(\mathbb{R}^3) := 
\left\{ 
f \in \mathcal{S}'_h(\mathbb{R}^3) \;\middle|\; 
\|f\|_{\dot B^{s}_{p,r}} < \infty 
\right\}$$
equipped with the norm, for $1\leq r<\infty$, \begin{align*}\|f\|_{\dot{B}^s_{p,r}}:=\left(\sum_{k\in\mathbb{Z}}2^{rsk}\|\dot{\Delta}_k f\|_{L^p}^r \right)^{\frac{1}{r}}, \end{align*} and for $r = \infty$, 
\begin{align*}
\|f\|_{\dot{B}^s_{p,\infty}}:=\sup_{k\in\mathbb{Z}}2^{sk}\|\dot{\Delta}_k f\|_{L^p}. 
\end{align*}
For later use, Bony's homogeneous decomposition is written as
\begin{align*}
 fg=T_fg+T_gf+R(f,g),
\end{align*}
where
\begin{align*}
 T_fg:=\sum_{k\in\mathbb Z}\dot S_{k-1}f\,\dot\Delta_kg,
 \qquad
 R(f,g):=\sum_{k\in\mathbb Z}\dot\Delta_kf\,
 \widetilde{\dot\Delta}_kg.
\end{align*}
For more details on Besov spaces, see \cite{BCD-Book-2011}.

\subsection{Useful lemmas}
We now collect several lemmas and estimates that will be used repeatedly below.
We begin with several standard properties of homogeneous Besov spaces.
\begin{lem}\label{L2.2}{\rm(\!\!\cite[Chapter 2]{BCD-Book-2011})}
The following properties hold:
\begin{itemize}
\item For any nonnegative integer $k$ and $s\in\mathbb{R}$, the following norm equivalence holds: $\|\nabla^k f\|_{\dot B_{p,r}^s}\sim \|f\|_{\dot B_{p,r}^{s + k}}$. In particular, when $p=r = 2$, we have $\|f\|_{\dot B_{2,2}^{k}}\sim \|f\|_{\dot H^k}$.
\item For $s\in\mathbb{R}$ and $p_1,p_2,r_1,r_2$ satisfying $1\leq p_{1}\leq p_{2}\leq\infty$ and $1\leq r_{1}\leq r_{2}\leq\infty$, one has the continuous embedding
\begin{align*}
\dot{B}^{s}_{p_{1},r_{1}}\hookrightarrow \dot{B}^{s-3 (\frac{1}{p_1}-\frac{1}{p_2})}_{p_{2},r_{2}}.
\end{align*}
\item If $1\leq p\leq q\leq\infty$, then
\begin{align*}
\dot{B}^{0}_{p,1}\hookrightarrow L^{p}\hookrightarrow \dot{B}^{0}_{p,\infty}\hookrightarrow \dot B_{q,\infty}^{\varsigma},\quad \text{where }\varsigma=-3\Big(\frac{1}{p}-\frac{1}{q}\Big).
\end{align*}
\item If $p < \infty$, then $\dot{B}^{\frac{3}{p}}_{p,1}$ is continuously embedded in the space of continuous functions that vanish at infinity.
\item For $1\leq p,r\leq\infty$, $s_1<s_2$, and $\theta\in(0,1)$, the following real interpolation estimate holds:
\begin{align*}
\|f\|_{\dot B_{p,r}^{\theta s_1+(1 - \theta)s_2}}\lesssim \frac{1}{\theta(1 - \theta)(s_2 - s_1)}\|f\|_{\dot B_{p,\infty}^{s_1}}^{\theta} \|f\|_{\dot B_{p,\infty}^{s_2}}^{1 - \theta}.
\end{align*}

\item{}
For $\sigma\in \mathbb{R}$ and $f\in{\mathcal S}^{'}_{h}(\mathbb{R}^3)$, define $\Lambda^{\sigma}f:=(-\Delta )^{\frac{\sigma}{2}}f:=\mathcal{F}^{-1}\big{(} |\xi|^{\sigma}\mathcal{F}(f) \big{)}$. Then $\Lambda^{\sigma}$ is an isomorphism from $\dot{B}^{s}_{p,r}$ to $\dot{B}^{s-\sigma}_{p,r}$.
\item Let $G$ be smooth on a neighborhood of the range of $f$, with $G(0)=0$. If $\sigma>0$, $f\in L^\infty\cap\dot B^\sigma_{2,r}$, and $1\leq r\leq\infty$, then
\begin{align*}
 \|G(f)\|_{\dot B^\sigma_{2,r}}
 \lesssim_{G,\|f\|_{L^\infty}}\|f\|_{\dot B^\sigma_{2,r}}.
\end{align*}
At the endpoint $\sigma=0$, if
$f\in\dot B^0_{2,\infty}\cap\dot B^{\frac32}_{2,1}$, then
\begin{align*}
 \|G(f)\|_{\dot B^0_{2,\infty}}
 \lesssim_{G,\|f\|_{L^\infty}}
 \|f\|_{\dot B^0_{2,\infty}}.
\end{align*}
\item{} Let $1\leq p_{1},p_{2},r_{1},r_{2}\leq \infty$, $s_{1}\in\mathbb{R}$ and $s_{2}\in\mathbb{R}$ satisfy
\begin{align*}
 s_{2}<\frac{3}{p_2}\quad\text{or}\quad s_{2}=\frac{3}{p_2}~\text{and}~r_{2}=1.
\end{align*}
Then the space $\dot{B}^{s_{1}}_{p_{1},r_{1}}\cap \dot{B}^{s_{2}}_{p_{2},r_{2}}$, endowed with the norm $\|\cdot \|_{\dot{B}^{s_{1}}_{p_{1},r_{1}}}+\|\cdot\|_{\dot{B}^{s_{2}}_{p_{2},r_{2}}}$, is a Banach space and enjoys the weak compactness and Fatou properties: if $f_{n}$ is a uniformly bounded sequence in $\dot{B}^{s_{1}}_{p_{1},r_{1}}\cap \dot{B}^{s_{2}}_{p_{2},r_{2}}$, then there exist an element $f$ of $\dot{B}^{s_{1}}_{p_{1},r_{1}}\!\cap \dot{B}^{s_{2}}_{p_{2},r_{2}}$ and a subsequence $f_{n_{k}}$ such that $f_{n_{k}}\rightarrow f$ in $\mathcal{S}'$ and
    \begin{align*}
    \begin{aligned}
    \|f\|_{\dot{B}^{s_{1}}_{p_{1},r_{1}}\cap \dot{B}^{s_{2}}_{p_{2},r_{2}}}\lesssim \liminf_{k\rightarrow \infty} \|f_{n_{k}}\|_{\dot{B}^{s_{1}}_{p_{1},r_{1}}\cap \dot{B}^{s_{2}}_{p_{2},r_{2}}}.
    \end{aligned}
    \end{align*}
\end{itemize}
\end{lem}

The following lemma establishes bilinear product estimates in homogeneous Besov spaces.
\begin{lem}[{\!\!\cite[Theorems 2.47 and 2.52]{BCD-Book-2011}}]\label{L2.1}
Let $s_1, s_2\in \mathbb{R}$ satisfy $s_1,s_2<\frac{3}{2}$ and $s_1 + s_2>0$, and let $1\leq r_1,r_2\leq\infty$ satisfy $\frac{1}{r_1}+\frac{1}{r_2}=\frac{1}{r}$. Then
\begin{align*}
\|fg\|_{\dot B_{2,r}^{s_1+s_2-\frac{3}{2}}}\lesssim_{s_1,s_2}\|f\|_{\dot B_{2,r_1}^{s_1}}\|g\|_{\dot B_{2,r_2}^{s_2}}.
\end{align*}
If $s_1\leq \frac{3}{2}$, $s_2<\frac{3}{2}$, and $s_1 + s_2\geq 0$, then
\begin{align*}
\|fg\|_{\dot B_{2,\infty}^{s_1+s_2-\frac{3}{2}}}\lesssim \|f\|_{\dot B_{2,1}^{s_1}}\|g\|_{\dot B_{2,\infty}^{s_2}}.
\end{align*}
\end{lem}

We next record product estimates in homogeneous Sobolev spaces in Lemmas \ref{L2.3}--\ref{L2.4} and in inhomogeneous Sobolev spaces in Lemma \ref{L2.5}.
\begin{lem}[{\!\!\cite[Proposition A.29]{BV-2022}}]\label{L2.3}
Let $f,g\in\dot H^{s}(\mathbb{R}^3)$ for $s\in(0,1)$. Then
\begin{align*}
\|fg\|_{\dot H^s}\lesssim\|f\|_{\dot  H^s}\|g\|_{L^\infty}+\|f\|_{L^\infty}\|g\|_{\dot H^s}.
\end{align*}
\end{lem}

\begin{lem}[{\!\!\cite[Proposition A.30]{BV-2022}}]\label{L2.4}
For $f,g\in \mathcal{S}(\mathbb{R}^3)$, the following estimate holds:
\begin{align*}
\|fg\|_{\dot H^{\frac{1}{2}}}\lesssim \|f\|_{\dot H^1}\|g\|_{\dot H^1}.
\end{align*}
\end{lem}

\begin{lem}[{\!\!\cite{AF-Pa-2003}}]\label{L2.5}
For any $f,g\in H^3(\mathbb{R}^3)$ and any multi-index $\alpha$ with $1\leq|\alpha|\leq3$, the following estimates hold:
\begin{align*}
\|f\|_{L^{\infty} }\lesssim&\,     \|\nabla  f\|_{L^{2} }^{\frac{1}{2}}
\|\nabla ^{2}f\|_{L^{2} }^{\frac{1}{2}}\lesssim \|f\|_{H^2}, \\
\|\partial ^{\alpha}(fg)\|_{L^{2} }
\lesssim&\,  \|\nabla  f\|_{H^{2} }\|\nabla  g\|_{H^{2} },\\
\|f\|_{L^3 }\lesssim&\, \|  f\|_{\dot H^{\frac{1}{2}} },\\
\|f\|_{L^6 }\lesssim&\, \|  f\|_{\dot H^1 },\\
\|f\|_{L^q }\lesssim&\,\|f\|_{H^1 },\quad 2\leq q\leq 6.
\end{align*}
\end{lem}

Finally, we record a commutator estimate that will be used in energy estimates.
\begin{lem}[{\!\!\cite[Appendix]{commutator1,commutator2}}]\label{L2.6}
Let $f$ and $g$ be Schwartz functions. For $k\geq0$,
\begin{align*}
\|\nabla^{k}(fg) \|_{L^r} \lesssim&\, \|f\|_{L^{r_1} }\|\nabla^{k}g\|_{L^{r_2} }+ \|g\|_{L^{r_3} }\|\nabla^{k}f\|_{L^{r_4}}.
\end{align*}
and for $k\geq1$,
\begin{align*}
\|\nabla^{k}(fg)-f\nabla^k g \|_{L^r} \lesssim&\, \|\nabla f\|_{L^{r_1}}\|\nabla^{k-1}g\|_{L^{r_2}}+ \|g\|_{L^{r_3}}\|\nabla^{k}f\|_{L^{r_4}},
\end{align*}
where $1<r,r_2,r_4<\infty$ and $r_i$ $(1\leq i\leq 4)$ satisfy
\begin{align*}
\frac{1}{r_1}+\frac{1}{r_2}=\frac{1}{r_3}+\frac{1}{r_4}=\frac{1}{r}.
\end{align*}
\end{lem}

\section{Global well-posedness of strong solutions}\label{S3}

In this section, we prove Theorem~\ref{T1.1} by direct energy estimates for the perturbation system \eqref{I3}. We first propagate the two homogeneous endpoint norms $\dot H^{\frac12-\delta}$ and $\dot H^3$. We then combine these estimates with the full $H^3$ energy inequality, in the spirit of \cite[Proposition~5.6]{Deguchi-MA-2024}, without requiring the $L^2$ norm of the initial perturbation to be small.

\subsection{A priori estimates}
Let $(\varrho,\omega)$ be a strong solution of \eqref{I3} on $[0,T]$ satisfying the bootstrap assumption
\begin{align}\label{G3.1}
 \sup_{0\leq t\leq T}
 \|(\varrho,\omega)(t)\|_{\dot H^{\frac12-\delta}\cap\dot H^3}
 \leq \sigma,
\end{align}
where $0<\sigma<1$ will be chosen sufficiently small. By the stationary relation \eqref{A1} and assumption \eqref{G1.7}, we have
\begin{align}\label{G3.2}
 \|\bar\rho\|_{\dot B^{\frac12}_{2,1}\cap\dot B^{\frac92}_{2,1}}
 \lesssim
 \|\phi\|_{\dot B^{\frac12}_{2,1}\cap\dot B^{\frac92}_{2,1}}
 \leq \varepsilon_0. 
\end{align}
Interpolation between the two endpoint bounds in \eqref{G3.2} gives
\begin{align*}
 \|\bar\rho\|_{\dot B^r_{2,1}}\lesssim\varepsilon_0,
 \qquad \frac12\leq r\leq\frac92.
\end{align*}
By Lemma~\ref{L2.2}, we obtain
\begin{align}\label{G3.4}
 &\|\bar\rho\|_{\dot H^{\frac12}\cap\dot H^4}
 +\|\bar\rho\|_{\dot B^{\frac32}_{2,1}\cap\dot B^{\frac52}_{2,1}}
 +\|\bar\rho\|_{W^{3,\infty}}
 \lesssim\varepsilon_0.
\end{align}
Since $\dot H^s(\mathbb R^3)=\dot B^s_{2,2}(\mathbb R^3)\hookrightarrow\dot B^s_{2,\infty}(\mathbb R^3)$, Lemma~\ref{L2.2} also yields, for $\frac12-\delta<r<3$,
\begin{align*}
 \|f\|_{\dot B^r_{2,1}}
 \lesssim
 \|f\|_{\dot H^{\frac12-\delta}}^{\frac{3-r}{\frac52+\delta}}
 \|f\|_{\dot H^3}^{\frac{r-\frac12+\delta}{\frac52+\delta}}.
\end{align*}
Then for $0\leq t\leq T$, it holds that
\begin{align}\label{G3.6}
 \|(\varrho,\omega)\|_{\dot H^{\frac12}}
 +\|(\varrho,\omega)\|_{\dot B^{\frac32}_{2,1}\cap\dot B^{\frac52}_{2,1}}
 \lesssim\sigma.
\end{align}
After choosing $\varepsilon_0+\sigma$ sufficiently small, \eqref{G3.4} and \eqref{G3.6} imply that
\begin{align}\label{G3.7}
 \frac{\rho_\infty}{2}
 \leq \varrho(t,x)+\rho_*(x)
 \leq \frac{3\rho_\infty}{2},
 \qquad (t,x)\in[0,T]\times\mathbb R^3.
\end{align}
By \eqref{G3.7}, the smooth coefficients in \eqref{I4} are uniformly bounded. As in \cite[Lemma~3.1]{DLUY-JDE-2007}, the nonlinearities have the schematic forms
\begin{align*}
 \mathcal N_1\sim &\,
 \partial_i\varrho\,\omega^i
 +\varrho\,\partial_i\omega^i
 +\partial_i\bar\rho\,\omega^i
 +\bar\rho\,\partial_i\omega^i,\\
\nonumber
 \mathcal N_2\sim&\,
 \omega^i\partial_i\omega^j
 +\varrho\,\partial_i\partial_i\omega^j
 +\varrho\,\partial_j\partial_i\omega^i
 +\varrho\,\partial_j\varrho\nonumber\\
 &+\bar\rho\,\partial_i\partial_i\omega^j
 +\bar\rho\,\partial_j\partial_i\omega^i
 +\partial_j\bar\rho\,\varrho
 +\bar\rho\,\partial_j\varrho.\nonumber
\end{align*}
Here $1\leq i,j\leq 3$.

We start with the homogeneous estimate of the endpoint case with regularity exponent $\frac12-\delta$, which provides crucial damping control for the density perturbation $\varrho$.
\begin{lem}[The $\dot H^{\frac12-\delta}$ estimate]
For the strong solution $(\varrho,\omega)$ to the Cauchy problem \eqref{I3}, there exists a constant $\lambda_1>0$ such that
\begin{align}\label{G3.10}
 &\frac{\mathrm d}{\mathrm dt}
 \|(\varrho,\omega)\|_{\dot H^{\frac12-\delta}}^2
 +\lambda_1\|\omega\|_{\dot H^{\frac32-\delta}}^2 \leq
 C(\varepsilon_0+\sigma)
 \big(
 \|\varrho\|_{\dot H^{\frac32-\delta}}^2
 +\|\omega\|_{\dot H^{\frac32-\delta}}^2
 \big),
\end{align}
and
\begin{align}\label{G3.11}
 &\frac{\mathrm d}{\mathrm dt}
 \big\langle\Lambda^{\frac12-\delta}\nabla\varrho,
 \Lambda^{\frac12-\delta}\omega\big\rangle
 +\lambda_1\|\varrho\|_{\dot H^{\frac32-\delta}}^2\nonumber\\
 \leq&\,
 C\|\omega\|_{\dot H^{\frac32-\delta}}^2
 +C\|\omega\|_{\dot H^{\frac52-\delta}}^2
 +C(\varepsilon_0+\sigma)
 \big(
 \|\varrho\|_{\dot H^{\frac32-\delta}}^2
 +\|\omega\|_{\dot H^{\frac32-\delta}}^2
 +\|\varrho\|_{\dot H^3}^2
 +\|\omega\|_{\dot H^4}^2
 \big).
\end{align}
\end{lem}

\begin{proof}
Applying $\Lambda^{\frac12-\delta}$ to \eqref{I3}$_1$ and \eqref{I3}$_2$, taking the $L^2$ inner products of the resulting equations with $\Lambda^{\frac12-\delta}\varrho$ and $\Lambda^{\frac12-\delta}\omega$, respectively, and adding the two identities, we obtain
\begin{align}\label{G3.12}
 &\frac12\frac{\mathrm d}{\mathrm dt}
 \big(\|\varrho\|_{\dot H^{\frac12-\delta}}^2
 +\|\omega\|_{\dot H^{\frac12-\delta}}^2\big)
 +\mu_1\|\nabla\omega\|_{\dot H^{\frac12-\delta}}^2
 +\mu_2\|\dive\omega\|_{\dot H^{\frac12-\delta}}^2\nonumber\\
 =&\,
 \big\langle\Lambda^{\frac12-\delta}\mathcal N_1,
 \Lambda^{\frac12-\delta}\varrho\big\rangle
 +\big\langle\Lambda^{\frac12-\delta}\mathcal N_2,
 \Lambda^{\frac12-\delta}\omega\big\rangle
 =:I_1+I_2.
\end{align}
For brevity, denote the dissipation appearing below by
\begin{align*}
 \mathcal D(t):=
 \|\varrho\|_{\dot H^{\frac32-\delta}}^2
 +\|\omega\|_{\dot H^{\frac32-\delta}}^2
 +\|\varrho\|_{\dot H^3}^2
 +\|\omega\|_{\dot H^4}^2.
\end{align*}
By \eqref{G3.1}--\eqref{G3.6}, all coefficients in the nonlinear
estimates below are bounded by $C(\varepsilon_0+\sigma)$.

For the pressure terms, define
\begin{align*}
 F(z):=\frac{P'(z)}{z},\qquad
 \mathcal P(\varrho,\bar\rho)
 :=\int_{\rho_*}^{\rho_*+\varrho}F(z)\,\mathrm dz
 -F(\rho_\infty)\varrho.
\end{align*}
Simple calculation gives
\begin{align}\label{G3.18}
 \nabla\mathcal P={}&
 \big(F(\rho_*+\varrho)-F(\rho_*)\big)\nabla\bar\rho
 +\big(F(\rho_*+\varrho)-F(\rho_\infty)\big)\nabla\varrho,
\end{align}
and
\begin{align}\label{G3.19}
 \mathcal P(\varrho,\bar\rho)
 =\bar\rho\,\varrho\,p_1(\bar\rho)
 +\varrho^2p_2(\varrho,\bar\rho),
\end{align}
where
\begin{align*}
 p_1(b)&:=\int_0^1F'(\rho_\infty+\theta b)\,\mathrm d\theta,\\
 p_2(v,b)&:=\int_0^1(1-\theta)
 F'(\rho_\infty+b+\theta v)\,\mathrm d\theta.
\end{align*}
A direct product estimate yields
\begin{align}\label{G3.20}
 \|\mathcal P\|_{\dot H^{\frac12-\delta}}
 &\lesssim
 \big(\|\bar\rho\|_{\dot H^{\frac12}}
 +\|\varrho\|_{\dot H^{\frac12}}\big)
 \|\varrho\|_{\dot H^{\frac32-\delta}},\nonumber\\
 \|\mathcal P\|_{\dot H^{\frac32-\delta}}
 &\lesssim
 \big(\|\bar\rho\|_{\dot B^{\frac32}_{2,1}}
 +\|\varrho\|_{\dot B^{\frac32}_{2,1}}\big)
 \|\varrho\|_{\dot H^{\frac32-\delta}}.
\end{align}
Moreover, for
\begin{align*}
 h(\varrho,\bar\rho)
 :=\frac{\varrho+\bar\rho}{\varrho+\rho_*}
 =\frac{\varrho+\bar\rho}{\rho_\infty+\varrho+\bar\rho},
\end{align*}
we have
\begin{align}\label{G3.22}
 \|h\|_{\dot B^{\frac32}_{2,1}}
 \lesssim
 \|\varrho\|_{\dot B^{\frac32}_{2,1}}
 +\|\bar\rho\|_{\dot B^{\frac32}_{2,1}}.
\end{align}

Since $\mathcal N_1=-\frac{\mu_1\gamma}{\mu}\dive[(\varrho+\bar\rho)\omega]$, Lemma \ref{L2.1} gives
\begin{align*}
 \|\mathcal N_1\|_{\dot H^{-\frac12-\delta}}
 &\lesssim
 \|(\varrho+\bar\rho)\omega\|_{\dot H^{\frac12-\delta}}\lesssim
 \big(\|\varrho\|_{\dot H^{\frac12}}
 +\|\bar\rho\|_{\dot H^{\frac12}}\big)
 \|\omega\|_{\dot H^{\frac32-\delta}},
\end{align*}
which, together with \eqref{G3.2}--\eqref{G3.6}, yields
\begin{align}\label{G3.24}
 |I_1|
 &\leq
 \|\mathcal N_1\|_{\dot H^{-\frac12-\delta}}
 \|\varrho\|_{\dot H^{\frac32-\delta}} \lesssim
 (\varepsilon_0+\sigma)
 \big(
 \|\varrho\|_{\dot H^{\frac32-\delta}}^2
 +\|\omega\|_{\dot H^{\frac32-\delta}}^2
 \big).
\end{align}
By \eqref{G3.18}, \eqref{I4}$_2$ can be written exactly as
\begin{align}\label{G3.25}
 \mathcal N_2={}&
 -\frac{\mu_1\gamma}{\mu}\omega\cdot\nabla\omega
 -\mu_1h\Delta\omega
 -\mu_2h\nabla\dive\omega
 -\frac{\mu}{\mu_1\gamma}\nabla\mathcal P.
\end{align}
The four terms on the right-hand side of \eqref{G3.25} satisfy
\begin{align}
 \|\omega\cdot\nabla\omega\|_{\dot H^{-\frac12-\delta}}
 &\lesssim
 \|\omega\|_{\dot H^{\frac12}}
 \|\omega\|_{\dot H^{\frac32-\delta}},\label{G3.26}\\
 \|h\Delta\omega\|_{\dot H^{-\frac12-\delta}}
 +\|h\nabla\dive\omega\|_{\dot H^{-\frac12-\delta}}
 &\lesssim
 \|h\|_{\dot B^{\frac32}_{2,1}}
 \|\omega\|_{\dot H^{\frac32-\delta}},\label{G3.26.1}\\
 \|\nabla\mathcal P\|_{\dot H^{-\frac12-\delta}}
 &=\|\mathcal P\|_{\dot H^{\frac12-\delta}}
 \lesssim
 \big(\|\bar\rho\|_{\dot H^{\frac12}}
 +\|\varrho\|_{\dot H^{\frac12}}\big)
 \|\varrho\|_{\dot H^{\frac32-\delta}}.\label{G3.26.2}
\end{align}
It follows from \eqref{G3.26}--\eqref{G3.26.2} that
\begin{align}\label{G3.27}
 |I_2|
 \leq
 \|\mathcal N_2\|_{\dot H^{-\frac12-\delta}}
 \|\omega\|_{\dot H^{\frac32-\delta}} \lesssim
 (\varepsilon_0+\sigma)
 \big(
 \|\varrho\|_{\dot H^{\frac32-\delta}}^2
 +\|\omega\|_{\dot H^{\frac32-\delta}}^2
 \big).
\end{align}
Substituting \eqref{G3.24} and \eqref{G3.27} into \eqref{G3.12}, and then using the smallness of $\varepsilon_0+\sigma$, we obtain \eqref{G3.10}.

We next recover the density dissipation. Applying $\Lambda^{\frac12-\delta}$ to \eqref{I3}$_2$, taking the inner product of the resulting equation with $\Lambda^{\frac12-\delta}\nabla\varrho$, and using \eqref{I3}$_1$, we have
\begin{align}\label{G3.28}
 &\frac{\mathrm d}{\mathrm dt}
 \big\langle\Lambda^{\frac12-\delta}\nabla\varrho,
 \Lambda^{\frac12-\delta}\omega\big\rangle
 +\gamma\|\varrho\|_{\dot H^{\frac32-\delta}}^2\nonumber\\
 =&\,
 \gamma\|\dive\omega\|_{\dot H^{\frac12-\delta}}^2
 +\mu_1\big\langle\Lambda^{\frac12-\delta}\Delta\omega,
 \Lambda^{\frac12-\delta}\nabla\varrho\big\rangle
 +\mu_2\big\langle\Lambda^{\frac12-\delta}\nabla\dive\omega,
 \Lambda^{\frac12-\delta}\nabla\varrho\big\rangle\nonumber\\
 & -\big\langle\Lambda^{\frac12-\delta}\mathcal N_1,
 \Lambda^{\frac12-\delta}\dive\omega\big\rangle+\big\langle\Lambda^{\frac12-\delta}\mathcal N_2,
 \Lambda^{\frac12-\delta}\nabla\varrho\big\rangle.
\end{align}
For any $\eta>0$, we have
\begin{align}\label{G3.29}
 \mu_1\big|\big\langle\Lambda^{\frac12-\delta}\Delta\omega,
 \Lambda^{\frac12-\delta}\nabla\varrho\big\rangle\big|
 +\mu_2\big|\big\langle\Lambda^{\frac12-\delta}\nabla\dive\omega,
 \Lambda^{\frac12-\delta}\nabla\varrho\big\rangle\big| \leq
 \eta\|\varrho\|_{\dot H^{\frac32-\delta}}^2
 +C_\eta\|\omega\|_{\dot H^{\frac52-\delta}}^2,
\end{align}
and $\mathcal N_1$ satisfies
\begin{align*}
 \|\mathcal N_1\|_{\dot H^{\frac12-\delta}}
  \lesssim
 \|(\varrho+\bar\rho)\omega\|_{\dot H^{\frac32-\delta}} \lesssim
 \big(\|\varrho\|_{\dot B^{\frac32}_{2,1}}
 +\|\bar\rho\|_{\dot B^{\frac32}_{2,1}}\big)
 \|\omega\|_{\dot H^{\frac32-\delta}}.
\end{align*}
For $\mathcal N_2$, the estimates in \eqref{G3.20} and \eqref{G3.22}, and the representation \eqref{G3.25} give
\begin{align*}
 \|\omega\cdot\nabla\omega\|_{\dot H^{\frac12-\delta}}
 &\lesssim
 \|\omega\|_{\dot B^{\frac32}_{2,1}}
 \|\omega\|_{\dot H^{\frac32-\delta}},\nonumber\\
 \|h\Delta\omega\|_{\dot H^{\frac12-\delta}}
 +\|h\nabla\dive\omega\|_{\dot H^{\frac12-\delta}}
 &\lesssim
 \|h\|_{\dot B^{\frac32}_{2,1}}
 \|\omega\|_{\dot H^{\frac52-\delta}}\lesssim \big(\|\bar\rho\|_{\dot B^{\frac32}_{2,1}}
 +\|\varrho\|_{\dot B^{\frac32}_{2,1}}\big) \|\omega\|_{\dot H^{\frac52-\delta}},\nonumber\\
 \|\nabla\mathcal P\|_{\dot H^{\frac12-\delta}}
 &=\|\mathcal P\|_{\dot H^{\frac32-\delta}} \lesssim
 \big(\|\bar\rho\|_{\dot B^{\frac32}_{2,1}}
 +\|\varrho\|_{\dot B^{\frac32}_{2,1}}\big)
 \|\varrho\|_{\dot H^{\frac32-\delta}}.
\end{align*}
Consequently,
\begin{align}\label{G3.32}
 &\big|\big\langle\Lambda^{\frac12-\delta}\mathcal N_1,
 \Lambda^{\frac12-\delta}\dive\omega\big\rangle\big|+\big|\big\langle\Lambda^{\frac12-\delta}\mathcal N_2,
 \Lambda^{\frac12-\delta}\nabla\varrho\big\rangle\big|
 \nonumber\\
 \leq&\,
 \eta\|\varrho\|_{\dot H^{\frac32-\delta}}^2
 +C_\eta\|\omega\|_{\dot H^{\frac32-\delta}}^2
 +C_\eta\|\omega\|_{\dot H^{\frac52-\delta}}^2
 +C(\varepsilon_0+\sigma)\mathcal D(t).
\end{align}
Substituting \eqref{G3.29} and \eqref{G3.32} into \eqref{G3.28} and choosing $\eta>0$ sufficiently small yield \eqref{G3.11}.
\end{proof}

With the low-order fractional homogeneous damping bound established above, we proceed to derive high-order a priori estimates in $\dot{H}^3$ for  $(\varrho,\omega)$, which together with the previous bound closes the homogeneous energy control.
\begin{lem}[The $\dot H^3$ estimate and closure]\label{L3.2} For the strong solution $(\varrho,\omega)$ to the Cauchy problem \eqref{I3}, if $\varepsilon_0 + \sigma$ is sufficiently small, then the following estimate holds on $[0,T]$:
\begin{align}\label{G3.33}
 &\sup_{0\leq t\leq T}
 \|(\varrho,\omega)(t)\|_{\dot H^{\frac12-\delta}\cap\dot H^3}^2+
 \int_0^T\big(
 \|\varrho(t)\|_{\dot H^{\frac32-\delta}}^2
 +\|\omega(t)\|_{\dot H^{\frac32-\delta}}^2
 +\|\varrho(t)\|_{\dot H^3}^2
 +\|\omega(t)\|_{\dot H^4}^2
 \big)\,\mathrm dt\nonumber\\
 &\leq C
 \|(\varrho_0,\omega_0)\|_{\dot H^{\frac12-\delta}\cap\dot H^3}^2.
\end{align}
\end{lem}

\begin{proof}
For any multi-index $\alpha$ with $|\alpha|=3$, applying $\partial^\alpha$ to $\eqref{I3}_1$ and $\eqref{I3}_2$, taking the inner products of the resulting equations with $\partial^\alpha\varrho$ and $\partial^\alpha\omega$, respectively,  and summing over all such $\alpha$, we obtain
\begin{align}\label{G3.34}
 &\frac12\frac{\mathrm d}{\mathrm dt}
 \sum_{|\alpha|=3}
 \big(\|\partial^\alpha\varrho\|_{L^2}^2
 +\|\partial^\alpha\omega\|_{L^2}^2\big)
 +\mu_1\|\omega\|_{\dot H^4}^2
 +\mu_2\|\dive\omega\|_{\dot H^3}^2\nonumber\\
 =&\,
 \sum_{|\alpha|=3}
 \big(
 \langle\partial^\alpha\mathcal N_1,\partial^\alpha\varrho\rangle
 +\langle\partial^\alpha\mathcal N_2,\partial^\alpha\omega\rangle
 \big)
 =:J_1+J_2.
\end{align}
The contribution of $\mathcal N_1$ can be decomposed as
\begin{align}\label{G3.35}
 J_1 &\sim\,
 \sum_{|\alpha|=3}
 \langle\partial^\alpha(\omega\cdot\nabla\varrho),
 \partial^\alpha\varrho\rangle
 +\sum_{|\alpha|=3}
 \langle\partial^\alpha(\varrho\dive\omega),
 \partial^\alpha\varrho\rangle +\sum_{|\alpha|=3}
 \langle\partial^\alpha\dive(\bar\rho\omega),
 \partial^\alpha\varrho\rangle\nonumber\\
 &=:J_{1,1}+J_{1,2}+J_{1,3}.
\end{align}
For the transport term in \eqref{G3.35}, Lemma \ref{L2.6} gives
\begin{align}\label{G3.36}
 |J_{1,1}|
 \lesssim&\,
 \frac12\|\dive\omega\|_{L^\infty}
 \|\varrho\|_{\dot H^3}^2
 +\sum_{|\alpha|=3}
 \|[\partial^\alpha,\omega\cdot\nabla]\varrho\|_{L^2}
 \|\partial^\alpha\varrho\|_{L^2}\nonumber\\
 \lesssim&\,
 \|\omega\|_{\dot B^{\frac52}_{2,1}}
 \|\varrho\|_{\dot H^3}^2
 +\|\varrho\|_{\dot B^{\frac52}_{2,1}}
 \|\omega\|_{\dot H^3}
 \|\varrho\|_{\dot H^3}.
\end{align}
Similarly, the second term in \eqref{G3.35} satisfies
\begin{align}
 |J_{1,2}|
 &\leq
 \|\varrho\|_{L^\infty}
 \|\omega\|_{\dot H^4}
 \|\varrho\|_{\dot H^3}
 +\sum_{|\alpha|=3}
 \|[\partial^\alpha,\varrho]\dive\omega\|_{L^2}
 \|\partial^\alpha\varrho\|_{L^2}\nonumber\\
 &\lesssim
 \|\varrho\|_{\dot B^{\frac32}_{2,1}}
 \|\omega\|_{\dot H^4}
 \|\varrho\|_{\dot H^3}
 +\Big(
 \|\varrho\|_{\dot B^{\frac52}_{2,1}}
 \|\omega\|_{\dot H^3}
 +\|\omega\|_{\dot B^{\frac52}_{2,1}}
 \|\varrho\|_{\dot H^3}
 \Big)\|\varrho\|_{\dot H^3}.\label{G3.36.1}
\end{align}
By Leibniz's rule and the product estimate in Lemma~\ref{L2.3},
\begin{align}
 |J_{1,3}|
 &\leq
 \|\dive(\bar\rho\omega)\|_{\dot H^3}
 \|\varrho\|_{\dot H^3}\nonumber\\
 &\lesssim
 \Big(
 \|\bar\rho\|_{L^\infty}\|\omega\|_{\dot H^4}
 +\|\bar\rho\|_{\dot H^4}\|\omega\|_{L^\infty}
 \Big)\|\varrho\|_{\dot H^3}\nonumber\\
 &\lesssim
 \Big(
 \|\bar\rho\|_{\dot B^{\frac32}_{2,1}}\|\omega\|_{\dot H^4}
 +\|\bar\rho\|_{\dot H^4}
 \|\omega\|_{\dot B^{\frac32}_{2,1}}
 \Big)\|\varrho\|_{\dot H^3}.\label{G3.36.2}
\end{align}
For any $\eta>0$, interpolation and Young's inequalities yield
\begin{align}
 \|\varrho\|_{\dot B^{\frac32}_{2,1}}^2
 &\lesssim
 \|\varrho\|_{\dot H^{\frac32-\delta}}^2
 +\|\varrho\|_{\dot H^3}^2,\label{G3.39.1}\\
 \|\omega\|_{\dot H^3}^2
 +\|\omega\|_{\dot B^{\frac32}_{2,1}}^2
 +\|\omega\|_{\dot B^{\frac52}_{2,1}}^2
 &\leq
 \eta\|\omega\|_{\dot H^4}^2
 +C_\eta\|\omega\|_{\dot H^{\frac32-\delta}}^2.\label{G3.39}
\end{align}
Combining \eqref{G3.36}--\eqref{G3.39} gives
\begin{align}\label{G3.40}
 |J_1|\lesssim(\varepsilon_0+\sigma)\mathcal D(t).
\end{align}

We next estimate $J_2$ on the right-hand side of \eqref{G3.34}. By \eqref{G3.25},
\begin{align*}
 J_2={}&J_{2,1}+J_{2,2}+J_{2,3}+J_{2,4},
\end{align*}
where the four terms correspond to
$\omega\cdot\nabla\omega$, $h\Delta\omega$,
$h\nabla\dive\omega$, and $\nabla\mathcal P$, respectively.
The convection term satisfies
\begin{align}\label{G3.42}
 |J_{2,1}|
 &\leq
 \frac12\|\dive\omega\|_{L^\infty}
 \|\omega\|_{\dot H^3}^2
 +\sum_{|\alpha|=3}
 \|[\partial^\alpha,\omega\cdot\nabla]\omega\|_{L^2}
 \|\partial^\alpha\omega\|_{L^2}\nonumber\\
 &\lesssim
 \|\omega\|_{\dot B^{\frac52}_{2,1}}
 \|\omega\|_{\dot H^3}^2.
\end{align}
For the variable viscosity terms, Leibniz's rule and integration by parts give
\begin{align}\label{G3.43}
 &\sum_{|\alpha|=3}\Big(
 \big|\langle\partial^\alpha(h\Delta\omega),
 \partial^\alpha\omega\rangle\big|
 +\big|\langle\partial^\alpha(h\nabla\dive\omega),
 \partial^\alpha\omega\rangle\big|\Big)\nonumber\\
 \lesssim&\,
 \big(\|h\|_{L^\infty}+\|\nabla h\|_{L^3}\big)
 \|\omega\|_{\dot H^4}^2
 +\|\nabla h\|_{L^\infty}
 \|\omega\|_{\dot H^4}\|\omega\|_{\dot H^3}\nonumber\\
 &+\|\nabla^2h\|_{L^3}
 \|\nabla^3\omega\|_{L^6}
 \|\omega\|_{\dot H^3}
 +\|\nabla^3h\|_{L^2}
 \|\nabla^2\omega\|_{L^3}
 \|\nabla^3\omega\|_{L^6}.
\end{align}
Standard composition estimates and \eqref{G3.4}--\eqref{G3.6} imply
\begin{align}\label{G3.44}
 &\|h\|_{L^\infty}
 +\|\nabla h\|_{L^3\cap L^\infty}
 +\|\nabla^2h\|_{L^3}
 +\|\nabla^3h\|_{L^2}
 \lesssim(\varepsilon_0+\sigma).
\end{align}
Since
\begin{align*}
 \|\nabla^3\omega\|_{L^6}\lesssim\|\omega\|_{\dot H^4},
 \qquad
 \|\nabla^2\omega\|_{L^3}
 \lesssim\|\omega\|_{\dot B^{\frac52}_{2,1}},
\end{align*}
we deduce from \eqref{G3.39}, \eqref{G3.43} and \eqref{G3.44} that
\begin{align}\label{G3.45}
 |J_{2,2}|+|J_{2,3}|
 \lesssim(\varepsilon_0+\sigma)\mathcal D(t).
\end{align}
For the pressure term, integration by parts gives
\begin{align*}
 |J_{2,4}|
 &=\bigg|
 \sum_{|\alpha|=3}
 \langle\partial^\alpha\mathcal P,
 \dive\partial^\alpha\omega\rangle
 \bigg|
 \leq
 \|\mathcal P\|_{\dot H^3}\|\omega\|_{\dot H^4}.
\end{align*}
By \eqref{G3.19}, Leibniz's rule, and the Sobolev embeddings in Lemma~\ref{L2.5}, we have
\begin{align*}
 \|\mathcal P\|_{\dot H^3}
 &\lesssim
 \big(\|\bar\rho\|_{L^\infty}
 +\|\varrho\|_{L^\infty}\big)
 \|\varrho\|_{\dot H^3}
 +\|\bar\rho\|_{\dot H^3}
 \|\varrho\|_{L^\infty}\\
 &\lesssim
 \big(\|\bar\rho\|_{\dot B^{\frac32}_{2,1}}
 +\|\varrho\|_{\dot B^{\frac32}_{2,1}}\big)
 \|\varrho\|_{\dot H^3}
 +\|\bar\rho\|_{\dot H^3}
 \|\varrho\|_{\dot B^{\frac32}_{2,1}}.
\end{align*}
Thus,
\begin{align}\label{G3.48}
 |J_{2,4}|\lesssim(\varepsilon_0+\sigma)\mathcal D(t).
\end{align}
Combining \eqref{G3.42}, \eqref{G3.45} and \eqref{G3.48} yields
\begin{align}\label{G3.49}
 |J_2|\lesssim(\varepsilon_0+\sigma)\mathcal D(t).
\end{align}
Substituting the estimates \eqref{G3.40} and \eqref{G3.49} into \eqref{G3.34} gives
\begin{align}\label{G3.50}
 \frac{\mathrm d}{\mathrm dt}
 \|(\varrho,\omega)\|_{\dot H^3}^2
 +\lambda_2\|\omega\|_{\dot H^4}^2
 \lesssim(\varepsilon_0+\sigma)\mathcal D(t).
\end{align}

To recover the density dissipation, similar to the calculation leading to \eqref{G3.28}, we obtain
\begin{align}\label{G3.51}
 &\frac{\mathrm d}{\mathrm dt}
 \sum_{|\alpha|=2}
 \langle\partial^\alpha\nabla\varrho,
 \partial^\alpha\omega\rangle
 +\gamma\|\varrho\|_{\dot H^3}^2\nonumber\\
 =&
 \gamma\|\dive\omega\|_{\dot H^2}^2
 +\mu_1\sum_{|\alpha|=2}
 \langle\partial^\alpha\Delta\omega,
 \partial^\alpha\nabla\varrho\rangle
 +\mu_2\sum_{|\alpha|=2}
 \langle\partial^\alpha\nabla\dive\omega,
 \partial^\alpha\nabla\varrho\rangle\nonumber\\
 &
 +\sum_{|\alpha|=2}
 \langle\partial^\alpha\mathcal N_2,
 \partial^\alpha\nabla\varrho\rangle
 -\sum_{|\alpha|=2}
 \langle\partial^\alpha\mathcal N_1,
 \partial^\alpha\dive\omega\rangle.
\end{align}
A direct calculation yields
\begin{align}\label{G3.52}
 \|\mathcal N_1\|_{\dot H^2}
 &\lesssim
 \|\varrho\|_{L^\infty}\|\omega\|_{\dot H^3}
 +\|\varrho\|_{\dot H^3}\|\omega\|_{L^\infty}
 +\|\bar\rho\|_{L^\infty}\|\omega\|_{\dot H^3}
 +\|\bar\rho\|_{\dot H^3}\|\omega\|_{L^\infty}\nonumber\\
 &\lesssim
 (\varepsilon_0+\sigma)
 \big(\|\varrho\|_{\dot H^3}
 +\|\omega\|_{\dot H^3}
 +\|\omega\|_{\dot B^{\frac32}_{2,1}}\big).
\end{align}
For the momentum nonlinearity, we also have
\begin{align*}
 \|\omega\cdot\nabla\omega\|_{\dot H^2}
 &\lesssim
 {\|\omega\|_{\dot B^{\frac32}_{2,1}}}
 \|\omega\|_{\dot H^3},\nonumber\\
 \|h\Delta\omega\|_{\dot H^2}
 +\|h\nabla\dive\omega\|_{\dot H^2}
 &\lesssim
 \|h\|_{L^\infty}\|\omega\|_{\dot H^4}
 +\|\nabla h\|_{L^\infty}\|\omega\|_{\dot H^3}
 +\|\nabla^2h\|_{L^3}\|\nabla^2\omega\|_{L^6}\nonumber\\
 &\lesssim
 (\varepsilon_0+\sigma)
 \big(\|\omega\|_{\dot H^3}+\|\omega\|_{\dot H^4}\big),\nonumber\\
 \|\nabla\mathcal P\|_{\dot H^2}
 &=\|\mathcal P\|_{\dot H^3} \lesssim
 (\varepsilon_0+\sigma)
 \big(\|\varrho\|_{\dot H^3}
 +\|\varrho\|_{\dot B^{\frac32}_{2,1}}\big).
\end{align*}
Hence,
\begin{align}\label{G3.54}
 \|\mathcal N_2\|_{\dot H^2}
 \lesssim
 (\varepsilon_0+\sigma)
 \big(\|\varrho\|_{\dot H^3}
 +\|\varrho\|_{\dot B^{\frac32}_{2,1}}
 +\|\omega\|_{\dot H^3}
 +\|\omega\|_{\dot H^4}\big).
\end{align}
Substituting \eqref{G3.52}--\eqref{G3.54} into \eqref{G3.51} and applying Young's inequality lead to
\begin{align}\label{G3.55}
 \frac{\mathrm d}{\mathrm dt}
 \sum_{|\alpha|=2}
 \langle\partial^\alpha\nabla\varrho,
 \partial^\alpha\omega\rangle
 +\lambda_2\|\varrho\|_{\dot H^3}^2 \leq
 C\|\omega\|_{\dot H^3}^2
 +C\|\omega\|_{\dot H^4}^2
 +C(\varepsilon_0+\sigma)\mathcal D(t),
\end{align}
after decreasing $\lambda_2>0$ if necessary.
For any $\eta>0$, it holds that
\begin{align}\label{G3.56}
 \|\omega\|_{\dot H^{\frac52-\delta}}^2
 +\|\omega\|_{\dot H^3}^2
 \leq
 \eta\|\omega\|_{\dot H^4}^2
 +C_\eta\|\omega\|_{\dot H^{\frac32-\delta}}^2.
\end{align}
Let $0<\eta_1\ll1$ and define
\begin{align*}
 \mathcal E(t):={}&
 \|(\varrho,\omega)\|_{\dot H^{\frac12-\delta}}^2
 +\|(\varrho,\omega)\|_{\dot H^3}^2 +\eta_1\big\langle\Lambda^{\frac12-\delta}\nabla\varrho,
 \Lambda^{\frac12-\delta}\omega\big\rangle
 +\eta_1\sum_{|\alpha|=2}
 \langle\partial^\alpha\nabla\varrho,
 \partial^\alpha\omega\rangle.
\end{align*}
By homogeneous interpolation and Cauchy's inequality, $\eta_1$ can be fixed sufficiently small so that
\begin{align}\label{G3.58}
 \mathcal E (t)
 \sim
 \|(\varrho,\omega)(t)\|_{\dot H^{\frac12-\delta}\cap\dot H^3}^2.
\end{align}
Taking a suitable linear combination of \eqref{G3.10}, \eqref{G3.11}, \eqref{G3.50} and \eqref{G3.55}, and using \eqref{G3.56}, we arrive at
\begin{align}\label{G3.59}
 \frac{\mathrm d}{\mathrm dt}\mathcal E(t)
 +\lambda_2\mathcal D(t)\leq0.
\end{align}
Integrating \eqref{G3.59} with respect to time, combined with \eqref{G3.58}, yields \eqref{G3.33}.
\end{proof}

Having established the low-order fractional damping bound and the high-order homogeneous energy control, we assemble these results to obtain the complete $H^3$ a priori estimate in the following lemma.
\begin{lem}[The energy estimate with a large $L^2$ part]\label{L3.3}
Assume that the homogeneous estimate \eqref{G3.33} holds. If $\varepsilon_0+\sigma$ is sufficiently small, then, for any $0\leq t\leq T$,
\begin{align}\label{G3.60}
  \|(\varrho,\omega)(t)\|_{H^3}^2
 +\int_0^t\big(
 \|\nabla\varrho(\tau)\|_{H^2}^2
 +\|\nabla\omega(\tau)\|_{H^3}^2
 \big)\,\mathrm d\tau \lesssim
 \|(\varrho_0,\omega_0)\|_{H^3}^2,
\end{align}
where no smallness condition is imposed on $\|(\varrho_0,\omega_0)\|_{L^2}$.
\end{lem}

\begin{proof}
Arguing analogously to \eqref{G3.34}, we have
\begin{align}\label{G3.61}
 &\frac12\frac{\mathrm d}{\mathrm dt}
 \sum_{|\alpha|\leq3}
 \big(\|\partial^\alpha\varrho\|_{L^2}^2
 +\|\partial^\alpha\omega\|_{L^2}^2\big)
 +\mu_1\|\nabla\omega\|_{H^3}^2
 +\mu_2\|\dive\omega\|_{H^3}^2\nonumber\\
 =&\,
 \sum_{|\alpha|\leq3}
 \big(
 \langle\partial^\alpha\mathcal N_1,\partial^\alpha\varrho\rangle
 +\langle\partial^\alpha\mathcal N_2,\partial^\alpha\omega\rangle
 \big).
\end{align}
For $|\alpha|\leq2$, the mixed energy identity is
\begin{align}\label{G3.62}
 &\frac{\mathrm d}{\mathrm dt}
 \sum_{|\alpha|\leq2}
 \langle\partial^\alpha\nabla\varrho,
 \partial^\alpha\omega\rangle
 +\gamma\|\nabla\varrho\|_{H^2}^2\nonumber\\
 =&\,
 \gamma\|\dive\omega\|_{H^2}^2
 +\mu_1\sum_{|\alpha|\leq2}
 \langle\partial^\alpha\Delta\omega,
 \partial^\alpha\nabla\varrho\rangle
 +\mu_2\sum_{|\alpha|\leq2}
 \langle\partial^\alpha\nabla\dive\omega,
 \partial^\alpha\nabla\varrho\rangle\nonumber\\
 &
 +\sum_{|\alpha|\leq2}
 \langle\partial^\alpha\mathcal N_2,
 \partial^\alpha\nabla\varrho\rangle
 -\sum_{|\alpha|\leq2}
 \langle\partial^\alpha\mathcal N_1,
 \partial^\alpha\dive\omega\rangle.
\end{align}
Set
\begin{align*}
 \mathcal D_3(t):=
 \|\nabla\varrho(t)\|_{H^2}^2
 +\|\nabla\omega(t)\|_{H^3}^2.
\end{align*}
By homogeneous interpolation, we derive
\begin{align*}
 &\|\varrho\|_{\dot B^{\frac32}_{2,1}\cap\dot B^{\frac52}_{2,1}}
 +\|\omega\|_{\dot H^3}
 +\|\omega\|_{\dot B^{\frac32}_{2,1}\cap\dot B^{\frac52}_{2,1}}
 \lesssim \sqrt{\mathcal D_3(t)}.
\end{align*}

We first estimate the zero-order terms in \eqref{G3.61}. The continuity equation gives
\begin{align*}
 \big|\langle\dive[(\varrho+\bar\rho)\omega],\varrho\rangle\big|
 &=\left|\int_{\mathbb R^3}(\varrho+\bar\rho)
 \omega\cdot\nabla\varrho\,\mathrm dx\right|\nonumber\\
 &\leq
 \big(\|\varrho\|_{L^3}+\|\bar\rho\|_{L^3}\big)
 \|\omega\|_{L^6}\|\nabla\varrho\|_{L^2}\nonumber\\
 &\lesssim
 \big(\|\varrho\|_{\dot H^{\frac12}}
 +\|\bar\rho\|_{\dot H^{\frac12}}\big)
 \|\nabla\omega\|_{L^2}\|\nabla\varrho\|_{L^2}.
\end{align*}
The convection term satisfies
\begin{align*}
 \big|\langle\omega\cdot\nabla\omega,\omega\rangle\big|
 &=\frac12\left|\int_{\mathbb R^3}
 (\dive\omega)|\omega|^2\,\mathrm dx\right|\nonumber\\
 &\leq
 \|\omega\|_{L^3}\|\nabla\omega\|_{L^2}\|\omega\|_{L^6}
 \lesssim
 \|\omega\|_{\dot H^{\frac12}}
 \|\nabla\omega\|_{L^2}^2.
\end{align*}
Integration by parts gives
\begin{align*}
 &|\langle h\Delta\omega,\omega\rangle|
 +|\langle h\nabla\dive\omega,\omega\rangle|\nonumber\\
 \lesssim&\,
 \|h\|_{L^\infty}\|\nabla\omega\|_{L^2}^2
 +\|\nabla h\|_{L^3}\|\omega\|_{L^6}
 \|\nabla\omega\|_{L^2}\nonumber\\
 \lesssim&\,
 \big(\|h\|_{L^\infty}+\|\nabla h\|_{L^3}\big)
 \|\nabla\omega\|_{L^2}^2.
\end{align*}
Finally, \eqref{G3.19} yields
\begin{align*}
 |\langle\nabla\mathcal P,\omega\rangle|
 &=|\langle\mathcal P,\dive\omega\rangle|\nonumber\\
 &\leq
 \|\mathcal P\|_{L^2}\|\nabla\omega\|_{L^2}\nonumber\\
 &\lesssim
 \big(\|\bar\rho\|_{L^3}+\|\varrho\|_{L^3}\big)
 \|\varrho\|_{L^6}\|\nabla\omega\|_{L^2}\nonumber\\
 &\lesssim
 \big(\|\bar\rho\|_{\dot H^{\frac12}}
 +\|\varrho\|_{\dot H^{\frac12}}\big)
 \|\nabla\varrho\|_{L^2}\|\nabla\omega\|_{L^2}.
\end{align*}
Thus, the zero-order contribution is bounded by $C(\varepsilon_0+\sigma)\mathcal D_3(t)$, and no undifferentiated $L^2$ norm occurs in the small coefficient.

For $1\leq|\alpha|\leq3$, the transport terms in the continuity equation satisfy
\begin{align}
 \sum_{1\leq|\alpha|\leq3}
 \big|\langle\partial^\alpha(\omega\cdot\nabla\varrho),
 \partial^\alpha\varrho\rangle\big| \lesssim&\,
 \|\omega\|_{\dot B^{\frac52}_{2,1}}
 \|\nabla\varrho\|_{H^2}^2
 +\|\varrho\|_{\dot B^{\frac52}_{2,1}}
 \|\nabla\omega\|_{H^2}\|\nabla\varrho\|_{H^2},\label{G3.69}\\
 \sum_{1\leq|\alpha|\leq3}
 \big|\langle\partial^\alpha(\varrho\dive\omega),
 \partial^\alpha\varrho\rangle\big| \lesssim&\,
 \|\varrho\|_{\dot B^{\frac32}_{2,1}}
 \|\nabla\omega\|_{H^3}\|\nabla\varrho\|_{H^2}\nonumber\\
 &
 +\Big(
 \|\varrho\|_{\dot B^{\frac52}_{2,1}}
 \|\nabla\omega\|_{H^2}
 +\|\omega\|_{\dot B^{\frac52}_{2,1}}
 \|\nabla\varrho\|_{H^2}
 \Big)\|\nabla\varrho\|_{H^2}.\label{G3.69.1}
\end{align}
The part involving stationary density in the continuity equation can be controlled by
\begin{align}
 \sum_{1\leq|\alpha|\leq3}
 \big|\langle\partial^\alpha\dive(\bar\rho\omega),
 \partial^\alpha\varrho\rangle\big| \lesssim&\,
 \Big(
 \|\bar\rho\|_{W^{3,\infty}}
 \|\nabla\omega\|_{H^3}
 +\|\nabla^4\bar\rho\|_{L^2}
 \|\omega\|_{L^\infty}
 \Big)\|\nabla\varrho\|_{H^2}\nonumber\\
  \lesssim&\,
 \Big(
 \|\bar\rho\|_{W^{3,\infty}}
 \|\nabla\omega\|_{H^3}
 +\|\nabla^4\bar\rho\|_{L^2}
 \|\omega\|_{\dot B^{\frac32}_{2,1}}
 \Big)\|\nabla\varrho\|_{H^2}.\label{G3.69.2}
\end{align}
For the convection term in the momentum equation, it holds that
\begin{align}
 \sum_{1\leq|\alpha|\leq3}
 \big|\langle\partial^\alpha(\omega\cdot\nabla\omega),
 \partial^\alpha\omega\rangle\big|
 \lesssim
 \|\omega\|_{\dot B^{\frac52}_{2,1}}
 \|\nabla\omega\|_{H^2}^2.\label{G3.69.3}
\end{align}
The calculation in \eqref{G3.43} gives
\begin{align}
 &\sum_{1\leq|\alpha|\leq3}\Big(
 \big|\langle\partial^\alpha(h\Delta\omega),
 \partial^\alpha\omega\rangle\big|
 +\big|\langle\partial^\alpha(h\nabla\dive\omega),
 \partial^\alpha\omega\rangle\big|\Big)\nonumber\\
 \lesssim&\,
 \Big(
 \|h\|_{L^\infty}
 +\|\nabla h\|_{L^3\cap L^\infty}
 +\|\nabla^2h\|_{L^3}
 +\|\nabla^3h\|_{L^2}
 \Big)\|\nabla\omega\|_{H^3}^2\nonumber\\
 \lesssim&\,(\varepsilon_0+\sigma)\|\nabla\omega\|_{H^3}^2.\label{G3.69.4}
\end{align}
For the pressure contribution, \eqref{G3.19} and Leibniz's rule give
\begin{align*}
 \|\nabla\mathcal P\|_{H^2}
 &\lesssim
 \Big(
 \|\bar\rho\|_{W^{3,\infty}}
 +\|\varrho\|_{\dot B^{\frac32}_{2,1}\cap\dot H^3}
 \Big)\|\nabla\varrho\|_{H^2}
 +\|\nabla\bar\rho\|_{H^2}
 \|\varrho\|_{\dot B^{\frac32}_{2,1}}\nonumber\\
 &\lesssim(\varepsilon_0+\sigma)\|\nabla\varrho\|_{H^2}.
\end{align*}
Therefore, we have
\begin{align}\label{G3.74}
 &\sum_{1\leq|\alpha|\leq3}
 \big|\langle\partial^\alpha\nabla\mathcal P,
 \partial^\alpha\omega\rangle\big|
 \leq
 \|\nabla\mathcal P\|_{H^2}
 \|\nabla\omega\|_{H^3}
 \lesssim(\varepsilon_0+\sigma)\mathcal D_3(t).
\end{align}
Combining \eqref{G3.69}--\eqref{G3.74} with \eqref{G3.4}, \eqref{G3.6}, and Young's inequality, we obtain
\begin{align}\label{G3.75}
 &\sum_{|\alpha|\leq3}
 \big(
 \big|\langle\partial^\alpha\mathcal N_1,
 \partial^\alpha\varrho\rangle\big|
 +\big|\langle\partial^\alpha\mathcal N_2,
 \partial^\alpha\omega\rangle\big|
 \big)
 \lesssim(\varepsilon_0+\sigma)\mathcal D_3(t).
\end{align}

It remains to estimate the nonlinear terms in \eqref{G3.62}. By Lemma \ref{L2.5} and \eqref{I4},
\begin{align}\label{G3.76}
 \|\mathcal N_1\|_{H^2}
  \lesssim
   {\|\nabla \omega\|_{H^2}(\|\nabla\bar{\rho}\|_{H^2}+\|\nabla\varrho\|_{H^2})}
 \lesssim
 (\varepsilon_0+\sigma)
 \big(\|\nabla\varrho\|_{H^2}
 +\|\nabla\omega\|_{H^2}\big).
\end{align}
The terms in $\mathcal N_2$ satisfy
\begin{align*}
 \|\omega\cdot\nabla\omega\|_{H^2}
 &\lesssim
  {\|\omega\|_{\dot B^{\frac32}_{2,1}\cap\dot B^{\frac52}_{2,1}}}
 \|\nabla\omega\|_{H^2},\nonumber\\
 \|h\Delta\omega\|_{H^2}
 +\|h\nabla\dive\omega\|_{H^2}
 &\lesssim
 \|h\|_{L^\infty}\|\nabla\omega\|_{H^3}
 +\|\nabla h\|_{L^\infty}\|\nabla\omega\|_{H^2}
 +\|\nabla^2h\|_{L^3}\|\nabla^2\omega\|_{L^6}\nonumber\\
 &\lesssim
 (\varepsilon_0+\sigma)\|\nabla\omega\|_{H^3},\nonumber\\
 \|\nabla\mathcal P\|_{H^2}
 &\lesssim
 (\varepsilon_0+\sigma)\|\nabla\varrho\|_{H^2}.
\end{align*}
Thus,  it holds that
\begin{align}\label{G3.78}
 \|\mathcal N_2\|_{H^2}
 \lesssim
 (\varepsilon_0+\sigma)
 \big(\|\nabla\varrho\|_{H^2}
 +\|\nabla\omega\|_{H^3}\big).
\end{align}
As a result,  from \eqref{G3.76} and \eqref{G3.78}, we get
\begin{align}\label{G3.79}
 &\sum_{|\alpha|\leq2}
 \big(
 \big|\langle\partial^\alpha\mathcal N_2,
 \partial^\alpha\nabla\varrho\rangle\big|
 +\big|\langle\partial^\alpha\mathcal N_1,
 \partial^\alpha\dive\omega\rangle\big|
 \big)
 \lesssim(\varepsilon_0+\sigma)\mathcal D_3(t).
\end{align}
Let $0<\eta_2\ll1$ and define
\begin{align*}
 \mathcal E_3(t):={}&
 \sum_{|\alpha|\leq3}
 \Big(\|\partial^\alpha\varrho\|_{L^2}^2
 +\|\partial^\alpha\omega\|_{L^2}^2\Big) +\eta_2\sum_{|\alpha|\leq2}
 \langle\partial^\alpha\nabla\varrho,
 \partial^\alpha\omega\rangle.
\end{align*}
For $\eta_2$ sufficiently small, we have
\begin{align}\label{G3.81}
 \mathcal E_3(t)
 \sim
 \|(\varrho,\omega)(t)\|_{H^3}^2.
\end{align}
Combining \eqref{G3.61}, \eqref{G3.62}, \eqref{G3.75} and \eqref{G3.79}, we obtain
\begin{align}\label{G3.82}
 \frac{\mathrm d}{\mathrm dt}\mathcal E_3(t)
 +\lambda_3\mathcal D_3(t)\leq0.
\end{align}
Then, integrating \eqref{G3.82} over $(0,t)$ and using \eqref{G3.81}, we get  \eqref{G3.60}.
\end{proof}

\subsection{Proof of Theorem~\ref{T1.1}}
Since the initial density is strictly positive and the stationary
coefficients belong to $W^{3,\infty}$, the standard local theory for
compressible viscous flows, applied to \eqref{I3},
provides a unique strong solution on a maximal interval $[0,T^*)$;
see, for example, \cite{CK-JDE-2003,CCK-JMPA-2004}. Let
\begin{align*}
 T^\sharp:=\sup\Big\{T\in(0,T^*):
 \sup_{0\leq t\leq T}
 \|(\varrho,\omega)(t)\|_{\dot H^{\frac12-\delta}\cap\dot H^3}
 \leq2C\varepsilon_0\Big\},
\end{align*}
where $C$ is the uniform constant in the homogeneous a priori estimate \eqref{G3.33}. If $\varepsilon_0$ is sufficiently small, Lemma~\ref{L3.2} improves the bootstrap bound to
\begin{align*}
 \sup_{0\leq t\leq T^\sharp}
 \|(\varrho,\omega)(t)\|_{\dot H^{\frac12-\delta}\cap\dot H^3}
 \leq C\varepsilon_0.
\end{align*}
By continuity, $T^\sharp=T^*$. In particular, the density remains uniformly positive on $[0,T^*)$.

Applying Lemma~\ref{L3.3} on $[0,T^*)$, we obtain
\begin{align*}
  \sup_{0\leq t<T^*}\|(\varrho,\omega)(t)\|_{H^3}^2
 +\int_0^{T^*}\Big(
 \|\nabla\varrho(t)\|_{H^2}^2
 +\|\nabla\omega(t)\|_{H^3}^2
 \Big)\,\mathrm dt \lesssim
 \|(\varrho_0,\omega_0)\|_{H^3}^2.
\end{align*}
The continuation criterion for strong solutions then gives
$T^*=\infty$.
Since $\omega$ is a fixed
positive multiple of $u$, the estimate \eqref{G1.8} follows, which
completes the proof of Theorem~\ref{T1.1}.\hfill $\Box$

\section{Optimal time-decay rates of the zeroth and first derivatives}\label{S4}

In this section, we prove Theorem~\ref{T1.2} directly for $(\varrho,\omega)$ in
\eqref{I3}. After the constants fixed in
Section~\ref{S3}, choose
\begin{align}\label{G4.1a}
 0<\eta_3\leq\frac1{16}
 \min\left\{1,\frac{\mu_1}{\gamma},
 \frac{\mu_1\gamma}{(\mu_1+\mu_2)^2}\right\},
\end{align}
and then fix
\begin{align}\label{G4.1}
 0<\lambda_4<\frac1{64}
 \min\big\{\lambda_3,\mu_1,\mu_1+\mu_2,\eta_3\gamma\big\}.
\end{align}
The same $\eta_3$ will be used in Section~\ref{S5}. We prove
\begin{align}\label{G4.2}
 \|\nabla^k(\varrho,\omega)(t)\|_{L^2}
 \lesssim (1+t)^{-\frac{k-s}{2}},\qquad k=0,1.
\end{align}

\subsection{Linear spectral estimate and low-frequency products}

Consider first the constant-coefficient homogeneous linear system
associated with \eqref{I3}:
\begin{equation}\label{G4.3}
 \left\{
 \begin{aligned}
  &\partial_t\varrho+\gamma\dive\omega=0,\\
  &\partial_t\omega-\mu_1\Delta\omega
   -\mu_2\nabla\dive\omega+\gamma\nabla\varrho=0,\\
  &(\varrho,\omega)|_{t=0}=(\varrho_0,\omega_0).
 \end{aligned}
 \right.
\end{equation}
We denote its solution by
$\mathbb S(t)(\varrho_0,\omega_0)$. Taking the Fourier transform of
\eqref{G4.3} gives
\begin{equation}\label{G4.4}
 \left\{
 \begin{aligned}
  &\partial_t\widehat\varrho
   +i\gamma\xi\cdot\widehat\omega=0,\\
  &\partial_t\widehat\omega
   +\mu_1|\xi|^2\widehat\omega
   +\mu_2\xi(\xi\cdot\widehat\omega)
   +i\gamma\xi\widehat\varrho=0.
 \end{aligned}
 \right.
\end{equation}
The transverse component satisfies
\begin{align*}
 \partial_t\left(I-\frac{\xi\otimes\xi}{|\xi|^2}\right)
 \widehat\omega
 +\mu_1|\xi|^2
 \left(I-\frac{\xi\otimes\xi}{|\xi|^2}\right)
 \widehat\omega=0.
\end{align*}
The characteristic polynomial of the longitudinal part is
\begin{align*}
 \lambda^2+(\mu_1+\mu_2)|\xi|^2\lambda
 +\gamma^2|\xi|^2=0,
\end{align*}
with roots
\begin{align*}
 \lambda_{\pm}(\xi)
 =-\frac{(\mu_1+\mu_2)|\xi|^2}{2}
 \pm\frac12\sqrt{(\mu_1+\mu_2)^2|\xi|^4
 -4\gamma^2|\xi|^2}.
\end{align*}
A direct energy calculation for \eqref{G4.4} gives
\begin{align}\label{G4.8}
 &\frac12\frac{\mathrm d}{\mathrm dt}
 \big(|\widehat\varrho|^2+|\widehat\omega|^2\big)
 +\mu_1|\xi|^2|\widehat\omega|^2
 +\mu_2|\xi\cdot\widehat\omega|^2=0,
\end{align}
and
\begin{align}\label{G4.9}
 &\frac{\mathrm d}{\mathrm dt}
 \operatorname{Re}(i\xi\widehat\varrho\mid\widehat\omega)
 +\gamma|\xi|^2|\widehat\varrho|^2=
 \gamma|\xi\cdot\widehat\omega|^2
 -(\mu_1+\mu_2)|\xi|^2
 \operatorname{Re}(i\xi\widehat\varrho\mid\widehat\omega).
\end{align}
With $\eta_3$ fixed in \eqref{G4.1a}, set
\begin{align*}
 \mathfrak F(\xi,t):={}&
 |\widehat\varrho(\xi,t)|^2
 +|\widehat\omega(\xi,t)|^2+\frac{\eta_3}{1+|\xi|^2}
 \operatorname{Re}
 (i\xi\widehat\varrho(\xi,t)\mid\widehat\omega(\xi,t)).
\end{align*}
Since
\begin{align*}
 \frac{|\xi|}{1+|\xi|^2}
 |\widehat\varrho||\widehat\omega|
 \leq\frac14
 \big(|\widehat\varrho|^2+|\widehat\omega|^2\big),
\end{align*}
we have
\begin{align*}
 \mathfrak F(\xi,t)\sim
 |\widehat\varrho(\xi,t)|^2
 +|\widehat\omega(\xi,t)|^2.
\end{align*}
Multiplying \eqref{G4.9} by
$\eta_3(1+|\xi|^2)^{-1}$ and adding the resulting inequality to twice
\eqref{G4.8}, we obtain
\begin{align*}
 \frac{\mathrm d}{\mathrm dt}\mathfrak F(\xi,t)
 +\lambda_4\frac{|\xi|^2}{1+|\xi|^2}
 \mathfrak F(\xi,t)
 \leq0.
\end{align*}
Indeed,
\begin{align*}
 \frac{\eta_3\gamma}{1+|\xi|^2}
 |\xi\cdot\widehat\omega|^2
 \leq&\,
 \frac{\mu_1}{4}|\xi|^2|\widehat\omega|^2,\\
 \frac{\eta_3(\mu_1+\mu_2)|\xi|^2}{1+|\xi|^2}
 \big|
 \operatorname{Re}(i\xi\widehat\varrho\mid\widehat\omega)
 \big|
 \leq&\,
 \frac{\eta_3\gamma}{4}
 \frac{|\xi|^2}{1+|\xi|^2}|\widehat\varrho|^2
 +\frac{\mu_1}{4}|\xi|^2|\widehat\omega|^2,
\end{align*}
provided $\eta_3$ is sufficiently small. Hence,
\begin{align}\label{G4.14}
 \big|
 \widehat{\mathbb S(t)(\varrho_0,\omega_0)}(\xi)
 \big|
 \lesssim
 \exp\left(
 -\frac{\lambda_4|\xi|^2t}{2(1+|\xi|^2)}
 \right)
 |(\widehat\varrho_0,\widehat\omega_0)(\xi)|.
\end{align}

Based on this frequency-wise exponential dissipation, we now establish the global low-frequency time-decay estimates for $\mathbb S(t)$ in the following lemma.
\begin{lem}[Low-frequency semigroup estimate]\label{L4.1}
Let $s\in[-\frac32,-1)$ and $k=0,1$. Then
\begin{align}\label{G4.15}
 \|\nabla^k[
 \mathbb S(t)(\varrho_0,\omega_0)]^L\|_{L^2}
 \lesssim
 (1+t)^{-\frac{k-s}{2}}
 \|(\varrho_0,\omega_0)\|_{\dot B^s_{2,\infty}}.
\end{align}
Moreover, for any source term $(F_1,F_2)$,
\begin{align}\label{G4.16}
 \left\|
 \nabla^k\left[
 \int_0^t
 \mathbb S(t-\tau)(F_1,F_2)(\tau)\,\mathrm d\tau
 \right]^L
 \right\|_{L^2}
 \lesssim
 \int_0^t
 (1+t-\tau)^{-\frac{k-s}{2}}
 \|(F_1,F_2)^L(\tau)\|_{\dot B^s_{2,\infty}}
 \,\mathrm d\tau.
\end{align}
\end{lem}

\begin{proof}
Let $j_0$ be fixed by the support of $\Phi_0$ in \eqref{G2.2}. From
\eqref{G4.14}, for $j\leq j_0$,
\begin{align*}
 \|\dot\Delta_j
 \mathbb S(t)(\varrho_0,\omega_0)\|_{L^2}
 \lesssim
 e^{-\lambda_4 2^{2j}t/4}
 \|\dot\Delta_j(\varrho_0,\omega_0)\|_{L^2}.
\end{align*}
Therefore,
\begin{align*}
 \|\nabla^k[
 \mathbb S(t)(\varrho_0,\omega_0)]^L\|_{L^2}^2
 \lesssim&\,
 \sum_{j\leq j_0}
 2^{2kj}e^{-\lambda_4 2^{2j}t/2}
 \|\dot\Delta_j(\varrho_0,\omega_0)\|_{L^2}^2\\
 \lesssim&\,
 \|(\varrho_0,\omega_0)\|_{\dot B^s_{2,\infty}}^2
 \sum_{j\leq j_0}
 2^{2(k-s)j}e^{-\lambda_4 2^{2j}t/2}\\
 \lesssim&\,
 (1+t)^{-(k-s)}
 \|(\varrho_0,\omega_0)\|_{\dot B^s_{2,\infty}}^2,
\end{align*}
where we use
\begin{align*}
 \sum_{j\leq j_0}
 2^{2(k-s)j}e^{-\lambda_4 2^{2j}t/2}
 \lesssim
 (1+t)^{-(k-s)}.
\end{align*}
The estimate \eqref{G4.16} follows from \eqref{G4.15} and Minkowski's
inequality.
\end{proof}
A dyadic splitting at frequency one gives
\begin{align*}
 \|\phi\|_{\dot B^{\frac32}_{2,1}}
 \lesssim_s
 \|\phi\|_{\dot B^{s+\frac32}_{2,\infty}}
 +\|\phi\|_{\dot B^{\frac92}_{2,1}}
 \lesssim\varepsilon,
\end{align*}
which in turn allows the $\sigma=0$ endpoint composition estimate  in Lemma~\ref{L2.2} to be applied when $s=-\frac32$. In addition, the stationary relation \eqref{A1},
the composition estimate in Lemma~\ref{L2.2}, and \eqref{G1.9} imply
\begin{align}\label{G4.20}
 &\left\|\left(
 \bar\rho,\frac{\bar\rho}{\rho_*},
 \bar\rho p_1(\bar\rho)\right)\right\|_{
 \dot B^{s+\frac32}_{2,\infty}\cap\dot B^{\frac92}_{2,1}}
 \lesssim\varepsilon,
\end{align}
and, by interpolation,
\begin{align*}
 &\left\|\left(
 \bar\rho,\frac{\bar\rho}{\rho_*},
 \bar\rho p_1(\bar\rho)\right)\right\|_{
 \dot B^{\frac32}_{2,1}\cap\dot B^{\frac52}_{2,1}}
 +\left\|\left(
 \bar\rho,\frac{\bar\rho}{\rho_*},
 \bar\rho p_1(\bar\rho)\right)\right\|_{W^{3,\infty}}
 \lesssim\varepsilon.
\end{align*}
We next give the low-frequency product estimate used below.

\begin{lem}[Low-frequency stationary products]\label{L4.2}
Under \eqref{G4.20},
\begin{align}
 \|[\dive(\bar\rho\,\omega)]^L\|_{\dot B^s_{2,\infty}}
 &\lesssim\varepsilon\|\nabla\omega\|_{L^2},\label{G4.22}\\
 \left\|\left[
 \frac{\bar\rho}{\rho_*}\Delta\omega
 \right]^L\right\|_{\dot B^s_{2,\infty}}
 +\left\|\left[
 \frac{\bar\rho}{\rho_*}\nabla\dive\omega
 \right]^L\right\|_{\dot B^s_{2,\infty}}
 &\lesssim\varepsilon\|\nabla\omega\|_{L^2},\label{G4.22.1}\\
\label{G4.24}
 \|[\nabla(\bar\rho p_1(\bar\rho)\varrho)]^L\|_{
 \dot B^s_{2,\infty}}
 &\lesssim\varepsilon\|\nabla\varrho\|_{L^2}.
\end{align}
\end{lem}

\begin{proof}
Let $a$ be a scalar coefficient satisfying
\begin{align}\label{G4.25}
 \|a\|_{\dot B^{s+\frac32}_{2,\infty}
 \cap\dot B^{\frac92}_{2,1}}\leq\varepsilon.
\end{align}
We first prove that, for any sufficiently regular scalar-valued
function $f$,
\begin{align}\label{G4.26}
 \|[\nabla(af)]^L\|_{\dot B^s_{2,\infty}}
 +\|[a\nabla^2f]^L\|_{\dot B^s_{2,\infty}}
 \lesssim\varepsilon\|\nabla f\|_{L^2}.
\end{align}
For a vector-valued function, \eqref{G4.26} is understood
componentwise.

By Bony's decomposition,
\begin{align*}
 af=T_af+T_fa+R(a,f).
\end{align*}
For $j\leq j_0$, \eqref{G4.25} gives
\begin{align*}
 \|\dot S_{j-1}a\|_{L^\infty}
 &\leq
 \sum_{\ell<j-1}2^{\frac32\ell}
 \|\dot\Delta_\ell a\|_{L^2}
 \lesssim
 2^{-sj}\|a\|_{\dot B^{s+\frac32}_{2,\infty}},\\
 \|\dot S_{j-1}f\|_{L^\infty}
 &\lesssim
 \sum_{\ell<j-1}2^{\frac32\ell}
 \|\dot\Delta_\ell f\|_{L^2}
 \lesssim2^{\frac12j}\|\nabla f\|_{L^2},\\
 \|\dot S_{j-1}\nabla^2f\|_{L^\infty}
 &\lesssim
 \sum_{\ell<j-1}2^{\frac72\ell}
 \|\dot\Delta_\ell f\|_{L^2}
 \lesssim2^{\frac52j}\|\nabla f\|_{L^2}.
\end{align*}
Consequently,
\begin{align}
 2^{sj}\|\dot\Delta_j\nabla(T_af)\|_{L^2}
 \lesssim&\,
{\varepsilon\|\nabla f\|_{L^2}}\label{G4.31}
\\
 2^{sj}\|\dot\Delta_j\nabla(T_fa)\|_{L^2}
 \lesssim&\,
 \varepsilon\|\nabla f\|_{L^2},\label{G4.31.1}\\
 2^{sj}\|\dot\Delta_j(T_a\nabla^2f)\|_{L^2}
 +2^{sj}\|\dot\Delta_j(T_{\nabla^2f}a)\|_{L^2}
 \lesssim&\,
 \varepsilon\|\nabla f\|_{L^2}.\label{G4.31.2}
\end{align}
For the remainders, the Fourier support properties guarantee that
\begin{align}
 2^{sj}\|\dot\Delta_j\nabla R(a,f)\|_{L^2}
 \lesssim&\,
 2^{(s+\frac52)j}
 \sum_{\ell\geq j-3}
 \|\dot\Delta_\ell a\|_{L^2}
 \|\widetilde{\dot\Delta}_\ell f\|_{L^2}\nonumber\\
 \lesssim&\,
 {\varepsilon
 \sum_{\ell\geq j-3
 }
 2^{(s+\frac52)(j-\ell)}
 2^\ell\|\widetilde{\dot\Delta}_\ell f\|_{L^2}}
 \lesssim\varepsilon\|\nabla f\|_{L^2},\label{G4.35.1}
\end{align}
and
\begin{align}\label{G4.35}
 2^{sj}\|\dot\Delta_jR(a,\nabla^2f)\|_{L^2}
 \lesssim&\,
 2^{(s+\frac32)j}
 \sum_{\ell\geq j-3}
 \|\dot\Delta_\ell a\|_{L^2}
 2^{2\ell}\|\widetilde{\dot\Delta}_\ell f\|_{L^2}\nonumber\\
 \lesssim&\,
 \varepsilon\left(
 \sum_{j-3\leq\ell\leq0}
 2^{(s+\frac32)(j-\ell)}2^\ell
 2^\ell\|\widetilde{\dot\Delta}_\ell f\|_{L^2}
 +\sum_{\ell\geq1}2^{-2\ell}
 2^\ell\|\widetilde{\dot\Delta}_\ell f\|_{L^2}
 \right)\nonumber\\
 \lesssim&\,\varepsilon\|\nabla f\|_{L^2}.
\end{align}
At the endpoint $s=-\frac32$, the first sum in \eqref{G4.35}
is bounded by Cauchy--Schwarz's inequality and
$\sum\limits_{\ell\leq0}2^{2\ell}<\infty$. Taking the supremum over
$j\leq j_0$ in \eqref{G4.31}--\eqref{G4.35} yields
\eqref{G4.26}.

By \eqref{G4.20}, the coefficients $ \bar\rho$, ${\bar\rho}/{\rho_*}$ and $\bar\rho p_1(\bar\rho)$
all satisfy \eqref{G4.25}. Applying the first estimate in
\eqref{G4.26} with $a=\bar\rho$ and $f=\omega_i$,
$i=1,2,3$, and summing over $i$, we obtain
\begin{align*}
 \|[\dive(\bar\rho\,\omega)]^L\|_{\dot B^s_{2,\infty}}
 &\leq
 \sum_{i=1}^3
 \|[\partial_i(\bar\rho\,\omega_i)]^L\|_{
 \dot B^s_{2,\infty}}\\
 &\lesssim
 \varepsilon\sum_{i=1}^3\|\nabla\omega_i\|_{L^2}
 \lesssim
 \varepsilon\|\nabla\omega\|_{L^2}.
\end{align*}
This proves the estimate in \eqref{G4.22}.

Next, apply the second estimate in \eqref{G4.26} with
$a=\bar\rho/\rho_*$ and $f=\omega_i$. Since every component of
$\Delta\omega$ and $\nabla\dive\omega$ is a linear combination of
second derivatives of the components $\omega_i$, summing the
componentwise estimates gives
\begin{align*}
 \left\|\left[
 \frac{\bar\rho}{\rho_*}\Delta\omega
 \right]^L\right\|_{\dot B^s_{2,\infty}}
 +\left\|\left[
 \frac{\bar\rho}{\rho_*}\nabla\dive\omega
 \right]^L\right\|_{\dot B^s_{2,\infty}}
 \lesssim
 \varepsilon\|\nabla\omega\|_{L^2}.
\end{align*}
The estimate in \eqref{G4.22.1} is thus proved.

Finally, applying the first estimate in \eqref{G4.26} with
\begin{align*}
 a=\bar\rho p_1(\bar\rho),
 \qquad
 f=\varrho,
\end{align*}
yields
\begin{align*}
 \|[\nabla(\bar\rho p_1(\bar\rho)\varrho)]^L\|_{\dot B^s_{2,\infty}}
 \lesssim
 \varepsilon\|\nabla\varrho\|_{L^2},
\end{align*}
which is \eqref{G4.24}.  The proof of Lemma \ref{L4.2} is completed.
\end{proof}
\subsection{Negative Besov norm and zeroth-order decay}

We first propagate the negative Besov norm. For $j\in\mathbb Z$,
applying $\dot\Delta_j$ to \eqref{I3}$_1$ and \eqref{I3}$_2$ gives
\begin{equation}\label{G4.18.1}
 \left\{
 \begin{aligned}
  &\partial_t\dot\Delta_j\varrho
   +\gamma\dive\dot\Delta_j\omega
   =\dot\Delta_j\mathcal N_1,\\
  &\partial_t\dot\Delta_j\omega
   -\mu_1\Delta\dot\Delta_j\omega
   -\mu_2\nabla\dive\dot\Delta_j\omega
   +\gamma\nabla\dot\Delta_j\varrho
   =\dot\Delta_j\mathcal N_2.
 \end{aligned}
 \right.
\end{equation}
Taking the $L^2$ inner products of $\eqref{G4.18.1}_1$ with
$\dot\Delta_j\varrho$ and of $\eqref{G4.18.1}_2$ with $\dot\Delta_j\omega$ yields
\begin{align}\label{G4.37}
 &\frac12\frac{\mathrm d}{\mathrm dt}
 \left(\|\dot\Delta_j\varrho\|_{L^2}^2
 +\|\dot\Delta_j\omega\|_{L^2}^2\right)
 +\mu_1\|\nabla\dot\Delta_j\omega\|_{L^2}^2
 +\mu_2\|\dive\dot\Delta_j\omega\|_{L^2}^2\nonumber\\
 =&
 \langle\dot\Delta_j\mathcal N_1,
 \dot\Delta_j\varrho\rangle
 +\langle\dot\Delta_j\mathcal N_2,
 \dot\Delta_j\omega\rangle.
\end{align}
Moreover,
\begin{align}\label{G4.38}
 &\frac{\mathrm d}{\mathrm dt}
 \langle\nabla\dot\Delta_j\varrho,
 \dot\Delta_j\omega\rangle
 +\gamma\|\nabla\dot\Delta_j\varrho\|_{L^2}^2\nonumber\\
 =&
 \gamma\|\dive\dot\Delta_j\omega\|_{L^2}^2
 +\mu_1\langle\nabla\dot\Delta_j\varrho,
 \Delta\dot\Delta_j\omega\rangle
 +\mu_2\langle\nabla\dot\Delta_j\varrho,
 \nabla\dive\dot\Delta_j\omega\rangle\nonumber\\
 &+
 \langle\nabla\dot\Delta_j\mathcal N_1,
 \dot\Delta_j\omega\rangle
 +\langle\nabla\dot\Delta_j\varrho,
 \dot\Delta_j\mathcal N_2\rangle.
\end{align}
Define
\begin{align}\label{G4.39}
 \mathscr E_j(t):={}&
 \|\dot\Delta_j\varrho(t)\|_{L^2}^2
 +\|\dot\Delta_j\omega(t)\|_{L^2}^2
 +\frac{\eta_3}{1+2^{2j}}
 \langle\nabla\dot\Delta_j\varrho(t),
 \dot\Delta_j\omega(t)\rangle.
\end{align}
For $\eta_3$ sufficiently small,
\begin{align}\label{G4.40}
 \mathscr E_j(t)\sim
 \|\dot\Delta_j(\varrho,\omega)(t)\|_{L^2}^2.
\end{align}
Combining \eqref{G4.37} and \eqref{G4.38}, and using Bernstein's
inequality, gives
\begin{align}\label{G4.41}
 &\frac{\mathrm d}{\mathrm dt}\mathscr E_j(t)
 +2\lambda_4\left(
 2^{2j}\|\dot\Delta_j\omega\|_{L^2}^2
 +\frac{2^{2j}}{1+2^{2j}}
 \|\dot\Delta_j\varrho\|_{L^2}^2\right)
 \leq|\mathscr Q_j(t)|,
\end{align}
where
\begin{align}\label{G4.42}
 \mathscr Q_j:={}&
 2\langle\dot\Delta_j\mathcal N_1,
 \dot\Delta_j\varrho\rangle
 +2\langle\dot\Delta_j\mathcal N_2,
 \dot\Delta_j\omega\rangle\nonumber\\
 &+\frac{\eta_3}{1+2^{2j}}
 \langle\nabla\dot\Delta_j\mathcal N_1,
 \dot\Delta_j\omega\rangle
 +\frac{\eta_3}{1+2^{2j}}
 \langle \dot\Delta_j\mathcal N_2,\nabla\dot\Delta_j\varrho
\rangle.
\end{align}
Using \eqref{G3.19}, \eqref{G3.22}, and the identity
\begin{align*}
 h(\varrho,\bar\rho)-\frac{\bar\rho}{\rho_*}
 =\frac{\rho_\infty\varrho}{
 \rho_*(\varrho+\rho_*)},
\end{align*}
we write \eqref{I4} exactly as
\begin{align}\label{G4.44}
 \mathcal N_1={}&
 -\frac{\mu_1\gamma}{\mu}\dive(\bar\rho\omega)
 -\frac{\mu_1\gamma}{\mu}\dive(\varrho\omega),\\
\label{G4.45}
 \mathcal N_2=&
 -\mu_1\frac{\bar\rho}{\rho_*}\Delta\omega
 -\mu_2\frac{\bar\rho}{\rho_*}\nabla\dive\omega
 -\frac{\mu}{\mu_1\gamma}
 \nabla\big(\bar\rho p_1(\bar\rho)\varrho\big)\nonumber\\
 &-\frac{\mu_1\gamma}{\mu}\omega\cdot\nabla\omega-\mu_1\frac{\rho_\infty\varrho}{
 \rho_*(\varrho+\rho_*)}\Delta\omega
 -\mu_2\frac{\rho_\infty\varrho}{
 \rho_*(\varrho+\rho_*)}\nabla\dive\omega\nonumber\\
 &-\frac{\mu}{\mu_1\gamma}
 \nabla\big(\varrho^2p_2(\varrho,\bar\rho)\big).
\end{align}
Let $\mathscr Q_j^{(1)}$ denote the terms in \eqref{G4.42} that are
linear in the stationary profile. We Apply Lemmas~\ref{L2.1} and
\ref{L2.6} to each nonlinear
product, then assemble the finitely many measurable
$\ell^2$ frequency envelopes. This yields a nonnegative measurable
sequence $(b_j(t))_{j\in\mathbb Z}$ satisfying
$\sum\limits_{j\in\mathbb Z} b_j(t)^2\leq1$, such that
\begin{align}
 2^{sj}\|\dot\Delta_j\dive(\varrho\omega)\|_{L^2}
 +2^{sj}\|\dot\Delta_j(\omega\cdot\nabla\omega)\|_{L^2}&\,\lesssim
 b_j(t)\,\sqrt{\mathcal D_3(t)}
 \|(\varrho,\omega)(t)\|_{\dot B^s_{2,\infty}},\label{G4.46}\\
 2^{sj}\left\|\dot\Delta_j\left(
 \frac{\rho_\infty\varrho}{\rho_*(\varrho+\rho_*)}
 (\Delta\omega,\nabla\dive\omega)\right)\right\|_{L^2}&\,\lesssim
 b_j(t)\,\sqrt{\mathcal D_3(t)}
 \|(\varrho,\omega)(t)\|_{\dot B^s_{2,\infty}},\label{G4.46.1}\\
 2^{sj}\left\|\dot\Delta_j\nabla
 \big(\varrho^2p_2(\varrho,\bar\rho)\big)\right\|_{L^2}&\,\lesssim
 b_j(t)\,\sqrt{\mathcal D_3(t)}
 \|(\varrho,\omega)(t)\|_{\dot B^s_{2,\infty}}.\label{G4.46.2}
\end{align}
For the viscosity products, we use
\begin{align}\label{G4.49}
 \dot\Delta_j(a\Delta\omega)
 =a\Delta\dot\Delta_j\omega
 +[\dot\Delta_j,a]\Delta\omega,
\end{align}
followed by
\begin{align}\label{G4.50}
 \langle a\Delta\dot\Delta_j\omega,
 \dot\Delta_j\omega\rangle
 =-\int_{\mathbb R^3}a|\nabla\dot\Delta_j\omega|^2\,\mathrm dx
 -\int_{\mathbb R^3}\nabla a\cdot\nabla\dot\Delta_j\omega
 \dot\Delta_j\omega\,\mathrm dx,
\end{align}
and
\begin{align}
 \left|\int_{\mathbb R^3}\nabla a\cdot
 \nabla\dot\Delta_j\omega\dot\Delta_j\omega\,\mathrm dx\right|
 \leq\|\nabla a\|_{L^3}
 \|\nabla\dot\Delta_j\omega\|_{L^2}
 \|\dot\Delta_j\omega\|_{L^6}
 \lesssim\|\nabla a\|_{L^3}
 \|\nabla\dot\Delta_j\omega\|_{L^2}^2.\label{G4.50.1}
\end{align}
The commutator term in \eqref{G4.49} satisfies
\begin{align}\label{G4.52}
 &2^{2sj}
 \big|\langle[\dot\Delta_j,a]\Delta\omega,
 \dot\Delta_j\omega\rangle\big|\leq
 \frac{\lambda_4}{64}2^{2sj+2j}
 \|\dot\Delta_j\omega\|_{L^2}^2
 +C b_j(t)^2\mathcal D_3(t)
 \|(\varrho,\omega)(t)\|_{\dot B^s_{2,\infty}}^2.
\end{align}
Here the coefficient $a$ stands either for the solution-dependent coefficient in the second line of \eqref{G4.45} or one of the stationary coefficients in the first line. For a stationary coefficient, the low-coefficient/high-solution
paraproduct and the commutator terms are absorbed as in
\eqref{G4.50}--\eqref{G4.52}. The high-coefficient/low-solution
paraproduct has to be retained. Indeed, for $j\leq j_0$, we have
\begin{align*}
 \|\dot S_{j-1}f\|_{L^6}
 &\lesssim\sum_{\ell<j-1}2^\ell
 \|\dot\Delta_\ell f\|_{L^2}
 \lesssim2^{(1-s)j}\|f\|_{\dot B^s_{2,\infty}},\nonumber\\
 \|\dot\Delta_ja\|_{L^3}
 &\lesssim2^{\frac j2}\|\dot\Delta_ja\|_{L^2}
 \lesssim\varepsilon 2^{-(1+s)j}.
\end{align*}
Consequently,
\begin{align}\label{G4.52b}
 2^{sj}\|\dot\Delta_j\nabla(T_fa)\|_{L^2}
 \lesssim\varepsilon 2^{(1-s)j}
 \|f\|_{\dot B^s_{2,\infty}},\qquad
 2^{sj}\|\dot\Delta_j(T_{\nabla^2f}a)\|_{L^2}
 \lesssim\varepsilon 2^{(2-s)j}
 \|f\|_{\dot B^s_{2,\infty}}.
\end{align}
The second inequality in \eqref{G4.52b} has one additional factor
$2^j$ and is bounded by the first one for $j\leq j_0$. Hence,
for $j\leq j_0$,
\begin{align}\label{G4.53}
 2^{2sj}|\mathscr Q_j^{(1)}|
 &\leq\frac{\lambda_4}{8}2^{2sj}
 \left(2^{2j}\|\dot\Delta_j\omega\|_{L^2}^2
 +\frac{2^{2j}}{1+2^{2j}}
 \|\dot\Delta_j\varrho\|_{L^2}^2\right)\nonumber\\
 &\quad+C\varepsilon\mathbf 1_{\{j\leq j_0\}}
 2^{(1-s)j}
 \|(\varrho,\omega)(t)\|_{\dot B^s_{2,\infty}}^2
 +Cb_j(t)^2\mathcal D_3(t)
 \|(\varrho,\omega)(t)\|_{\dot B^s_{2,\infty}}^2.
\end{align}
For $j>j_0$, the $\dot B^{\frac92}_{2,1}$ bound gives, for
any stationary coefficient $a$ in
\eqref{G4.44}--\eqref{G4.45},
\begin{align*}
 \|\dot\Delta_ja\|_{L^2}
 \lesssim\varepsilon 2^{-\frac92j}.
\end{align*}
For $j>j_0$, the low-frequency factor in the
high-coefficient/low-solution paraproduct is instead estimated by
\begin{align*}
 \|\dot S_{j-1}f\|_{L^\infty}
 &\lesssim 2^{(\frac32-s)j}
 \|f\|_{\dot B^s_{2,\infty}},\\
 \|\dot S_{j-1}\nabla^2f\|_{L^\infty}
 &\lesssim 2^{(\frac72-s)j}
 \|f\|_{\dot B^s_{2,\infty}}.
\end{align*}
Hence, by Bernstein's inequalities,
\begin{align*}
2^{sj}\left(
 \|\dot\Delta_j\nabla(T_fa)\|_{L^2}
 +\|\dot\Delta_j(T_{\nabla^2f}a)\|_{L^2}
 \right)\lesssim
 \varepsilon(2^{-2j}+2^{-j})
 \|f\|_{\dot B^s_{2,\infty}}
 \lesssim_{j_0}
 \varepsilon
 \|f\|_{\dot B^s_{2,\infty}},
 \qquad j>j_0.
\end{align*}
Here $f$ denotes either component of $(\varrho,\omega)$. All other nonlinear contributions can be bounded by identical arguments, while the
low-coefficient/high-solution interactions are absorbed by
\eqref{G4.50}--\eqref{G4.52}. Hence
\begin{align}\label{G4.53c}
 2^{2sj}|\mathscr Q_j^{(1)}|
 \leq\frac{\lambda_4}{64}2^{2sj}
 \left(2^{2j}\|\dot\Delta_j\omega\|_{L^2}^2
 +\frac{2^{2j}}{1+2^{2j}}
 \|\dot\Delta_j\varrho\|_{L^2}^2\right)+C\varepsilon
 \|(\varrho,\omega)(t)\|_{\dot B^s_{2,\infty}}^2,
\end{align}
for any $j>j_0$.
The estimate \eqref{G4.53} applies to $j\leq j_0$, and
\eqref{G4.53c} to $j>j_0$.
For the quadratic pressure term, integration by parts and
Lemma~\ref{L2.1} give
\begin{align}\label{G4.55}
 &2^{2sj}\big|\langle\dot\Delta_j\nabla
 (\varrho^2p_2(\varrho,\bar\rho)),
 \dot\Delta_j\omega\rangle\big|\nonumber\\
 \leq&\,
 \frac{\lambda_4}{64}2^{2sj+2j}
 \|\dot\Delta_j\omega\|_{L^2}^2
 +Cb_j(t)^2\mathcal D_3(t)
 \|(\varrho,\omega)(t)\|_{\dot B^s_{2,\infty}}^2.
\end{align}
The linear pressure interaction is contained in
$\mathscr Q_j^{(1)}$ and is controlled by \eqref{G4.53}.
Combining \eqref{G4.46}--\eqref{G4.46.2}, \eqref{G4.50}--\eqref{G4.55} with Young's inequality gives
\begin{align}\label{G4.56}
 2^{2sj}|\mathscr Q_j(t)|
 \leq&\,\lambda_4 2^{2sj}
 \left(2^{2j}\|\dot\Delta_j\omega\|_{L^2}^2
 +\frac{2^{2j}}{1+2^{2j}}
 \|\dot\Delta_j\varrho\|_{L^2}^2\right) +C\varepsilon\mathbf 1_{\{j\leq j_0\}}
 2^{(1-s)j}
 \|(\varrho,\omega)(t)\|_{\dot B^s_{2,\infty}}^2\nonumber\\
 & +C\varepsilon\mathbf 1_{\{j>j_0\}}
 \|(\varrho,\omega)(t)\|_{\dot B^s_{2,\infty}}^2
 +C b^2_j(t)\mathcal D_3(t)
 \|(\varrho,\omega)(t)\|_{\dot B^s_{2,\infty}}^2.
\end{align}
For $t\geq0$, set
\begin{align*}
 \mathcal E_{\infty}(t):=
 \sup_{0\leq\tau\leq t}
 \|(\varrho,\omega)(\tau)\|_{\dot B^s_{2,\infty}}.
\end{align*}
After lowering the fixed cutoff if necessary, we assume $j_0\leq0$.
It follows from \eqref{G4.40}--\eqref{G4.41} and \eqref{G4.56}
that, for $j\leq j_0$,
\begin{align}\label{G4.57}
 &\frac{\mathrm d}{\mathrm dt}
 \big(2^{2sj}\mathscr E_j(t)\big)
 +\lambda_4 2^{2j}
 \big(2^{2sj}\mathscr E_j(t)\big)\leq
 C\varepsilon 2^{(1-s)j} \mathcal E^2_{\infty}(t)
 +Cb^2_j(t)\mathcal D_3(t) \mathcal E^2_{\infty}(t),
\end{align}
and, for $j>j_0$,
\begin{align}\label{G4.57a}
 &\frac{\mathrm d}{\mathrm dt}
 \big(2^{2sj}\mathscr E_j(t)\big)
 +\lambda_4\big(2^{2sj}\mathscr E_j(t)\big)\leq
 C\varepsilon \mathcal E^2_{\infty}(t)
 +Cb^2_j(t)\mathcal D_3(t) \mathcal E^2_{\infty}(t).
\end{align}
Integrating \eqref{G4.57} with the factor
$e^{\lambda_4 2^{2j}t}$ over $(0,t)$ gives
\begin{align}\label{G4.57b}
 2^{2sj}\mathscr E_j(t)
 \leq2^{2sj}\mathscr E_j(0)
 +C\varepsilon 2^{(-1-s)j} \mathcal E^2_{\infty}(t)+C\int_0^t b^2_j(\tau)\mathcal D_3(\tau)
  \mathcal E^2_{\infty}(\tau)\,\mathrm d\tau,
 \qquad j\leq j_0.
\end{align}
Because $-1-s>0$ and $j\leq j_0\leq0$, the middle coefficient in
\eqref{G4.57b} is uniformly bounded. 
Integrating \eqref{G4.57a} with the factor
$e^{\lambda_4t}$ gives, for $j>j_0$,
\begin{align*}
 2^{2sj}\mathscr E_j(t)
 \leq&\,
 e^{-\lambda_4t}2^{2sj}\mathscr E_j(0)
 +C\varepsilon\int_0^t
 e^{-\lambda_4(t-\tau)}
 \mathcal E_{\infty}(\tau)^2\,\mathrm d\tau+C\int_0^t
 e^{-\lambda_4(t-\tau)}
 b_j(\tau)^2\mathcal D_3(\tau)
 \mathcal E_{\infty}(\tau)^2\,\mathrm d\tau\\
 \leq&\,
 2^{2sj}\mathscr E_j(0)
 +C\varepsilon\mathcal E_{\infty}(t)^2
 +C\int_0^t
 b_j(\tau)^2\mathcal D_3(\tau)
 \mathcal E_{\infty}(\tau)^2\,\mathrm d\tau.
\end{align*}
Combining this estimate with \eqref{G4.57b}, taking the supremum
over the dyadic indices and over $0\leq\tau\leq t$, and using
\eqref{G4.40}, $b_j(\tau)^2\leq1$, and
\begin{align*}
 2^{(-1-s)j}\leq2^{(-1-s)j_0}\leq1,
 \qquad j\leq j_0\leq0,
\end{align*}
we obtain
\begin{align}\label{G4.57c}
 \mathcal E_{\infty}(t)^2
 \leq
 C\|(\varrho_0,\omega_0)\|_{\dot B^s_{2,\infty}}^2
 +C\varepsilon\mathcal E_{\infty}(t)^2+C\int_0^t\mathcal D_3(\tau)
 \mathcal E_{\infty}(\tau)^2\,\mathrm d\tau.
\end{align}
Choosing $\varepsilon_1>0$ such that
$C\varepsilon_1\leq\frac12$, we absorb the second term on the
right-hand side of \eqref{G4.57c}. Gronwall's inequality and \eqref{G3.60} then yield
\begin{align*}
 \mathcal E_{\infty}(t)^2
 &\leq
 C\|(\varrho_0,\omega_0)\|_{\dot B^s_{2,\infty}}^2
 \exp\left(
 C\int_0^t\mathcal D_3(\tau)\,\mathrm d\tau
 \right)\\
 &\leq
 C\|(\varrho_0,\omega_0)\|_{\dot B^s_{2,\infty}}^2
 \exp\left(
 C\|(\varrho_0,\omega_0)\|_{H^3}^2
 \right),
\end{align*}
which leads to
\begin{align}\label{G4.58}
 \sup_{t\geq0}
 \|(\varrho,\omega)(t)\|_{\dot B^s_{2,\infty}}
 \leq
 C\|(\varrho_0,\omega_0)\|_{\dot B^s_{2,\infty}}.
\end{align}

By \eqref{G3.82} and \eqref{G4.1},
\begin{align}\label{G4.59}
 \frac{\mathrm d}{\mathrm dt}\mathcal E_3(t)
 +4\lambda_4\mathcal D_3(t)\leq0,
 \qquad
 \mathcal E_3(t)\sim
 \|(\varrho,\omega)(t)\|_{H^3}^2.
\end{align}
Since
\begin{align*}
 \sum_{1\leq|\alpha|\leq3}
 \|\partial^\alpha(\varrho,\omega)(t)\|_{L^2}^2
 \lesssim\mathcal D_3(t),
\end{align*}
we have
\begin{align*}
 \mathcal E_3(t)
 \lesssim
 \|(\varrho,\omega)(t)\|_{L^2}^2
 +\mathcal D_3(t).
\end{align*}
On the other hand, interpolation between
$\dot B^s_{2,\infty}$ and $\dot H^1$, together with
\eqref{G4.58}, gives
\begin{align*}
 \|(\varrho,\omega)(t)\|_{L^2}^2
 \lesssim
 \|(\varrho,\omega)(t)\|_{\dot B^s_{2,\infty}}^{
 \frac{2}{1-s}}
 \|\nabla(\varrho,\omega)(t)\|_{L^2}^{
 \frac{-2s}{1-s}}\lesssim
 \mathcal D_3(t)^{\frac{-s}{1-s}}.
\end{align*}
Consequently,
\begin{align*}
 \mathcal E_3(t)
 \lesssim
 \mathcal D_3(t)^{\frac{-s}{1-s}}
 +\mathcal D_3(t).
\end{align*}
Since $\mathcal E_3(t)\leq\mathcal E_3(0)$ by \eqref{G4.59},
distinguishing the cases $\mathcal D_3(t)\leq1$ and
$\mathcal D_3(t)>1$ yields
\begin{align*}
 \mathcal D_3(t)
 \gtrsim
 \mathcal E_3(t)^{\frac{1-s}{-s}}
 =
 \mathcal E_3(t)^{1-\frac1s}.
\end{align*}
Substituting this estimate into \eqref{G4.59}, we obtain
\begin{align*}
 \frac{\mathrm d}{\mathrm dt}\mathcal E_3(t)
 +C\mathcal E_3(t)^{1-\frac1s}\leq0.
\end{align*}
Solving this differential inequality gives
\begin{align*}
 \mathcal E_3(t)\lesssim(1+t)^s.
\end{align*}
Therefore,
\begin{align}\label{G4.69}
 \|(\varrho,\omega)(t)\|_{H^3}
 \lesssim(1+t)^{\frac s2}.
\end{align}
Finally, integrating \eqref{G4.59} over $[t,\infty)$ gives
\begin{align}\label{G4.70}
 \int_t^\infty\mathcal D_3(\tau)\,\mathrm d\tau
 \lesssim\mathcal E_3(t)
 \lesssim(1+t)^s.
\end{align}
This proves \eqref{G4.2} for $k=0$.

\subsection{Low-frequency estimate for the first derivative}
To further control the evolution of first-order derivatives and close the full iteration argument, we perform low-frequency estimates for nonlinear source terms arising from \eqref{G4.44}--\eqref{G4.45}. The following lemma provides a key bound for the low-frequency part of the nonlinear source $(\mathcal N_1,\mathcal N_2)^L$, which will be essential for subsequent iterative energy closures.

\begin{lem}[Low-frequency source estimate]\label{L4.3}
For all $t\geq0$,
\begin{align}\label{G4.71}
 \|(\mathcal N_1,\mathcal N_2)^L(t)\|_{\dot B^s_{2,\infty}}
 \lesssim
 \varepsilon\|\nabla(\varrho,\omega)(t)\|_{L^2}
 +\|(\varrho,\omega)(t)\|_{L^2}\sqrt{\mathcal D_3(t)}.
\end{align}
\end{lem}

\begin{proof}
By \eqref{G4.22}--\eqref{G4.24}, the terms in the first lines of \eqref{G4.44}--\eqref{G4.45}
satisfy
\begin{align}\label{G4.72}
 &\left\|\left(
 \dive(\bar\rho\omega),
 \frac{\bar\rho}{\rho_*}\Delta\omega,
 \frac{\bar\rho}{\rho_*}\nabla\dive\omega,
 \nabla(\bar\rho p_1(\bar\rho)\varrho)
 \right)^L\right\|_{\dot B^s_{2,\infty}}
 \lesssim\varepsilon\|\nabla(\varrho,\omega)\|_{L^2}.
\end{align}
Set
\begin{align*}
 p=\frac6{3-2s},\qquad q=-\frac3s.
\end{align*}
Since $-3/2\leq s<-1$, we have
\begin{align*}
 1\leq p<\frac65,
 \qquad2\leq q\leq3,
 \qquad\frac1p=\frac12+\frac1q,
 \qquad L^p\hookrightarrow\dot B^s_{2,\infty}.
\end{align*}
Hence
\begin{align}
 \|\dive(\varrho\omega)\|_{\dot B^s_{2,\infty}}
 \lesssim&\,
 \|\varrho\|_{L^2}\|\nabla\omega\|_{L^q}
 +\|\omega\|_{L^2}\|\nabla\varrho\|_{L^q}\lesssim
 \|(\varrho,\omega)\|_{L^2}\mathcal D_3^{1/2},\label{G4.75}\\
 \|\omega\cdot\nabla\omega\|_{\dot B^s_{2,\infty}}
 \lesssim&\,
 \|\omega\|_{L^2}\|\nabla\omega\|_{L^q}
 \lesssim\|(\varrho,\omega)\|_{L^2}\mathcal D_3^{1/2},\label{G4.75.1}
\end{align}
where $\|\nabla(\varrho,\omega)\|_{L^q}$ is controlled by
$\mathcal D_3^{1/2}$ because $2\leq q\leq3$. By \eqref{G3.7},
\begin{align*}
 \left|\frac{\rho_\infty\varrho}{
 \rho_*(\varrho+\rho_*)}\right|\lesssim|\varrho|,
\end{align*}
and therefore
\begin{align}
 \left\|\frac{\rho_\infty\varrho}{
 \rho_*(\varrho+\rho_*)}
 (\Delta\omega,\nabla\dive\omega)\right\|_{
 \dot B^s_{2,\infty}}\lesssim
 \|\varrho\|_{L^2}\|\nabla^2\omega\|_{L^q}
 \lesssim\|(\varrho,\omega)\|_{L^2}\mathcal D_3^{1/2}.\label{G4.75.2}
\end{align}
Finally,
\begin{align}
 \nabla\big(\varrho^2p_2(\varrho,\bar\rho)\big)=
 \big(2\varrho p_2+\varrho^2\partial_1p_2\big)\nabla\varrho
 +\varrho^2\partial_2p_2\nabla\bar\rho.\nonumber
\end{align}
The first term above satisfies
\begin{align}
 \|\big(2\varrho p_2+\varrho^2\partial_1p_2\big)
 \nabla\varrho\|_{L^p}
 \lesssim\|\varrho\|_{L^2}\|\nabla\varrho\|_{L^q}
 \lesssim\|(\varrho,\omega)\|_{L^2}\mathcal D_3^{1/2}.\label{G4.79}
\end{align}
For the second term, let
\begin{align*}
 r=\frac6{-1-2s}\in[3,6).
\end{align*}
Since
\begin{align*}
 \frac1p=\frac12+\frac16+\frac1r,
 \qquad \|\nabla\bar\rho\|_{L^r}\lesssim\varepsilon,
\end{align*}
we obtain
\begin{align}\label{G4.83}
 \|\varrho^2\partial_2p_2\nabla\bar\rho\|_{L^p}
 \lesssim
 \|\varrho\|_{L^2}\|\varrho\|_{L^6}
 \|\nabla\bar\rho\|_{L^r}
 \lesssim
 \varepsilon\|\varrho\|_{L^2}\|\nabla\varrho\|_{L^2}.
\end{align}
Consequently, from \eqref{G4.72}--\eqref{G4.83}, we get
\eqref{G4.71}.
\end{proof}

Having obtained the uniform bound for the low-frequency part of the nonlinear source $(\mathcal N_1,\mathcal N_2)^L$, we express the solution to the Cauchy problem \eqref{I3} using Duhamel’s formula:
\begin{align*}
 (\varrho,\omega)(t)
 =\mathbb S(t)(\varrho_0,\omega_0)
 +\int_0^t\mathbb S(t-\tau)
 (\mathcal N_1,\mathcal N_2)(\tau)\,\mathrm d\tau.
\end{align*}
By Lemmas~\ref{L4.1} and \ref{L4.3},
\begin{align}\label{G4.85}
 \|\nabla(\varrho,\omega)^L(t)\|_{L^2}
\lesssim&\,
 (1+t)^{-\frac{1-s}{2}}
 \|(\varrho_0,\omega_0)\|_{\dot B^s_{2,\infty}}\nonumber\\
 & +\varepsilon\int_0^t
 (1+t-\tau)^{-\frac{1-s}{2}}
 \|\nabla(\varrho,\omega)(\tau)\|_{L^2}\,\mathrm d\tau\nonumber\\
 & +\int_0^t
 (1+t-\tau)^{-\frac{1-s}{2}}
 \|(\varrho,\omega)(\tau)\|_{L^2}
 \sqrt{\mathcal D_3(\tau)}\,\mathrm d\tau.
\end{align}
For $t\geq2$, \eqref{G3.60} and \eqref{G4.69} give
\begin{align}\label{G4.86}
 &\int_0^{t/2}(1+t-\tau)^{-\frac{1-s}{2}}
 \|(\varrho,\omega)(\tau)\|_{L^2}
 \sqrt{\mathcal D_3(\tau)}\,\mathrm d\tau\nonumber\\
 \lesssim&\,
 (1+t)^{-\frac{1-s}{2}}
 \left(\int_0^\infty(1+\tau)^s\,\mathrm d\tau\right)^{1/2}
 \left(\int_0^\infty\mathcal D_3(\tau)\,\mathrm d\tau\right)^{1/2}\nonumber\\
 \lesssim&\,(1+t)^{-\frac{1-s}{2}},
\end{align}
and, \eqref{G4.70} yields
\begin{align}
 &\int_{t/2}^{t}(1+t-\tau)^{-\frac{1-s}{2}}
 \|(\varrho,\omega)(\tau)\|_{L^2}
\sqrt{\mathcal D_3(\tau)}\,\mathrm d\tau\nonumber\\
 \lesssim&\,
 (1+t)^{\frac s2}
 \left(\int_{t/2}^{t}(1+t-\tau)^{-(1-s)}\,\mathrm d\tau\right)^{1/2}
 \left(\int_{t/2}^{t}\mathcal D_3(\tau)\,\mathrm d\tau\right)^{1/2}\nonumber\\
 \lesssim&\,(1+t)^s
 \lesssim(1+t)^{-\frac{1-s}{2}}.\label{G4.86.1}
\end{align}
where the last inequality follows from $s<-1$. Moreover,
\begin{align}\label{G4.88}
 \int_0^t(1+t-\tau)^{-\frac{1-s}{2}}
 (1+\tau)^{-\frac{1-s}{2}}\,\mathrm d\tau
 \lesssim(1+t)^{-\frac{1-s}{2}},
\end{align}
because $(1-s)/2>1$.

We next establish a crucial differential dissipation inequality for the high-frequency component. Once this high-frequency decay property is available, we shall combine it with integral estimates \eqref{G4.86}--\eqref{G4.88} and a weighted Gronwall's inequality to deduce the full low-frequency temporal decay.

\subsection{High-frequency Lyapunov estimate and completion of the proof}
We now turn our attention to the complementary high-frequency component, and construct a suitable Lyapunov functional to close the high-order energy estimate. Define
\begin{align*}
 \mathfrak E_1(t):=
 \|\nabla\varrho(t)\|_{L^2}^2
 +\|\nabla\omega(t)\|_{L^2}^2
 +\eta_4\langle\nabla\varrho^H(t),\omega^H(t)\rangle,
\end{align*}
where $0<\eta_4\ll1$. By \eqref{G2.3},
\begin{align}\label{G4.90}
 \mathfrak E_1(t)\sim
 \|\nabla(\varrho,\omega)(t)\|_{L^2}^2.
\end{align}
We establish a dissipative differential bound for $\mathfrak E_1(t)$ in the following lemma.
\begin{lem}[First-order high-frequency inequality]
If $\varepsilon\leq\varepsilon_1$, with $\varepsilon_1$ fixed
after \eqref{G4.57c}, then
\begin{align}\label{G4.91}
 \frac{\mathrm d}{\mathrm dt}\mathfrak E_1(t)
 +\lambda_4\mathfrak E_1(t)
 \lesssim\|\nabla(\varrho,\omega)^L(t)\|_{L^2}^2.
\end{align}
\end{lem}

\begin{proof}
Applying $\nabla$ to \eqref{I3}, and taking the $L^2$ inner products of $\eqref{I3}_1$ with $\nabla\varrho$ and of $\eqref{I3}_2$ with $\nabla\omega$, gives
\begin{align}\label{G4.92}
 \frac12\frac{\mathrm d}{\mathrm dt}
 \|\nabla(\varrho,\omega)\|_{L^2}^2
 +\mu_1\|\nabla^2\omega\|_{L^2}^2
 +\mu_2\|\nabla\dive\omega\|_{L^2}^2=
 \langle\nabla\mathcal N_1,\nabla\varrho\rangle
 +\langle\nabla\mathcal N_2,\nabla\omega\rangle.
\end{align}
Applying $\Phi_1(D_x)$ to \eqref{I3}, and repeating the calculation
in \eqref{G4.38}, yields
\begin{align}\label{G4.93}
 &\frac{\mathrm d}{\mathrm dt}
 \langle\nabla\varrho^H,\omega^H\rangle
 +\gamma\|\nabla\varrho^H\|_{L^2}^2\nonumber\\
 =&\,
 \gamma\|\dive\omega^H\|_{L^2}^2
 +\mu_1\langle\nabla\varrho^H,\Delta\omega^H\rangle
 +\mu_2\langle\nabla\varrho^H,\nabla\dive\omega^H\rangle\nonumber\\
 &+
 \langle\nabla(\mathcal N_1)^H,\omega^H\rangle
 +\langle\nabla\varrho^H,(\mathcal N_2)^H\rangle.
\end{align}
The linear cross inner-product terms on the right-hand side of \eqref{G4.93}
satisfy
\begin{align}\label{G4.94}
 \eta_4\left|
 \mu_1\langle\nabla\varrho^H,\Delta\omega^H\rangle
 +\mu_2\langle\nabla\varrho^H,
 \nabla\dive\omega^H\rangle\right|\leq
 \frac{\lambda_4}{8}\|\nabla\varrho^H\|_{L^2}^2
 +\frac{\mu_1}{8}\|\nabla^2\omega\|_{L^2}^2.
\end{align}

We next estimate all terms in \eqref{G4.44}--\eqref{G4.45}.
For the stationary term in the continuity equation, we have
\begin{align*}
 \big|\langle\nabla\dive(\bar\rho\omega),
 \nabla\varrho\rangle\big|\lesssim&\,
 \Big(
 \|\bar\rho\|_{L^\infty}\|\nabla^2\omega\|_{L^2}
 +\|\nabla\bar\rho\|_{L^\infty}\|\nabla\omega\|_{L^2}
 +\|\nabla^2\bar\rho\|_{L^3}\|\omega\|_{L^6}
 \Big)\|\nabla\varrho\|_{L^2}\nonumber\\
 \leq&\,
 \frac{\lambda_4}{32}\|\nabla^2\omega\|_{L^2}^2
 +C\varepsilon\|\nabla(\varrho,\omega)\|_{L^2}^2.
\end{align*}
For $a=\bar\rho/\rho_*$, it holds that
\begin{align*}
 &\big|\langle\nabla(a\Delta\omega),
 \nabla\omega\rangle\big|
 +\big|\langle\nabla(a\nabla\dive\omega),
 \nabla\omega\rangle\big|
 \lesssim
 {\|a\|_{L^\infty}\|\nabla^2\omega\|_{L^2}^2}
 \lesssim\varepsilon\|\nabla^2\omega\|_{L^2}^2.
\end{align*}
For the linear pressure term, 
\begin{align*}
 \big|\langle\nabla^2(\bar\rho p_1(\bar\rho)\varrho),
 \nabla\omega\rangle\big|
 =&\,
 \big|\langle\nabla(\bar\rho p_1(\bar\rho)\varrho),
 \nabla\dive\omega\rangle\big|\nonumber\\
  \lesssim&\,
 \left(
 \|\bar\rho p_1(\bar\rho)\|_{L^\infty}\|\nabla\varrho\|_{L^2}
 +\|\nabla(\bar\rho p_1(\bar\rho))\|_{L^3}
 \|\varrho\|_{L^6}\right)
 \|\nabla^2\omega\|_{L^2}\nonumber\\
  \leq&\,
 \frac{\lambda_4}{32}\|\nabla^2\omega\|_{L^2}^2
 +C\varepsilon^2\|\nabla\varrho\|_{L^2}^2.
\end{align*}
The corresponding terms in \eqref{G4.93} satisfy
\begin{align*}
 \big|\langle\nabla[\dive(\bar\rho\omega)]^H,
 \omega^H\rangle\big|
 =\big|\langle[\dive(\bar\rho\omega)]^H,
 \dive\omega^H\rangle\big|\lesssim
 \varepsilon\|\nabla\omega\|_{L^2}
 \|\nabla\omega^H\|_{L^2},
\end{align*}
\begin{align*}
 \left\|\left(
 \frac{\bar\rho}{\rho_*}\Delta\omega,
 \frac{\bar\rho}{\rho_*}\nabla\dive\omega,
 \nabla(\bar\rho p_1(\bar\rho)\varrho)
 \right)^H\right\|_{L^2}\lesssim
 \varepsilon\big(
 \|\nabla^2\omega\|_{L^2}
 +\|\nabla\varrho\|_{L^2}\big).
\end{align*}
By virtue of the high-frequency Bernstein inequalities,
\begin{align}\label{G4.100}
 \|\nabla(\varrho,\omega)^H\|_{L^2}^2
 \lesssim
 \|\nabla\varrho^H\|_{L^2}^2
 +\|\nabla^2\omega^H\|_{L^2}^2.
\end{align}
Using \eqref{G4.94} and the estimates above, the 
stationary-profile terms in \eqref{G4.92}--\eqref{G4.93} are
bounded by
\begin{align*}
 \frac{\lambda_4}{8}
 \left(\|\nabla\varrho^H\|_{L^2}^2
 +\|\nabla^2\omega\|_{L^2}^2\right)
 +C\|\nabla(\varrho,\omega)^L\|_{L^2}^2.
\end{align*}

For the nonlinear term in the continuity equation, integration by parts yields
\begin{align}\label{G4.102}
 &\big|\langle\nabla\dive(\varrho\omega),
 \nabla\varrho\rangle\big|\nonumber\\
 \lesssim&\,
 \|\nabla\omega\|_{L^\infty}\|\nabla\varrho\|_{L^2}^2
 +\|\varrho\|_{L^\infty}\|\nabla^2\omega\|_{L^2}
 \|\nabla\varrho\|_{L^2}
 +\|\nabla\varrho\|_{L^3}\|\nabla\omega\|_{L^6}
 \|\nabla\varrho\|_{L^2}.
\end{align}
The convection term satisfies
\begin{align}\label{G4.102.1}
 \big|\langle\nabla(\omega\cdot\nabla\omega),
 \nabla\omega\rangle\big|
 \lesssim
 \|\nabla\omega\|_{L^\infty}
 \|\nabla\omega\|_{L^2}^2.
\end{align}
For
\begin{align*}
 a(\varrho,\bar\rho)
 =\frac{\rho_\infty\varrho}{
 \rho_*(\varrho+\rho_*)},
\end{align*}
by \eqref{G3.7}, we have
\begin{align*}
 \|a(\varrho,\bar\rho)\|_{L^\infty}
 +\|\nabla a(\varrho,\bar\rho)\|_{L^3\cap L^\infty}
 +\|\nabla^2 a(\varrho,\bar\rho)\|_{L^3}
 \lesssim\varepsilon,
\end{align*}
which further implies
\begin{align}\label{G4.102.2}
 &\big|\langle\nabla(a(\varrho,\bar\rho)\Delta\omega),
 \nabla\omega\rangle\big|
 +\big|\langle\nabla(a(\varrho,\bar\rho)
 \nabla\dive\omega),\nabla\omega\rangle\big|
 \lesssim\varepsilon\|\nabla^2\omega\|_{L^2}^2.
\end{align}
For the quadratic pressure term, \eqref{G4.79} and \eqref{G4.83} together yield
\begin{align}\label{G4.102.3}
 \big|\langle\nabla^2(\varrho^2p_2(\varrho,\bar\rho)),
 \nabla\omega\rangle\big|
 =\,
 \big|\langle\nabla(\varrho^2p_2(\varrho,\bar\rho)),
 \nabla\dive\omega\rangle\big|
 \lesssim\,
 \varepsilon\|\nabla\varrho\|_{L^2}
 \|\nabla^2\omega\|_{L^2}.
\end{align}
The nonlinear terms in \eqref{G4.93} satisfy
\begin{align}
 \|\dive(\varrho\omega)\|_{L^2}
&\lesssim\varepsilon\|\nabla(\varrho,\omega)\|_{L^2},\label{G4.102.4}\\
\label{G4.109}
 \left\|\left(
 \omega\cdot\nabla\omega,
 a(\varrho,\bar\rho)\Delta\omega,
 a(\varrho,\bar\rho)\nabla\dive\omega,
 \nabla(\varrho^2p_2(\varrho,\bar\rho))
 \right)\right\|_{L^2}
 &\lesssim
 \varepsilon\big(
 \|\nabla(\varrho,\omega)\|_{L^2}
 +\|\nabla^2\omega\|_{L^2}\big).
\end{align}
Making use of \eqref{G3.6}, \eqref{G4.100}, and Young's inequality to treat bounds \eqref{G4.102}--\eqref{G4.109}, all nonlinear contributions can be controlled by
\begin{align*}
 \frac{\lambda_4}{8}
 \left(\|\nabla\varrho^H\|_{L^2}^2
 +\|\nabla^2\omega\|_{L^2}^2\right)
 +C\|\nabla(\varrho,\omega)^L\|_{L^2}^2.
\end{align*}
Multiplying \eqref{G4.93} by $\eta_4$ and adding the resulting identity to twice equation \eqref{G4.92}, we obtain
\begin{align}\label{G4.111}
 \frac{\mathrm d}{\mathrm dt}\mathfrak E_1(t)
 +2\lambda_4\left(
 \|\nabla\varrho^H(t)\|_{L^2}^2
 +\|\nabla^2\omega(t)\|_{L^2}^2\right)
 \lesssim\|\nabla(\varrho,\omega)^L(t)\|_{L^2}^2.
\end{align}
Lastly, by frequency decomposition, we have
\begin{align}\label{G4.112}
 \|\nabla(\varrho,\omega)\|_{L^2}^2
 \lesssim
 \|\nabla(\varrho,\omega)^L\|_{L^2}^2
 +\|\nabla\varrho^H\|_{L^2}^2
 +\|\nabla^2\omega^H\|_{L^2}^2.
\end{align}
Adding a sufficiently small multiple of
$\|\nabla(\varrho,\omega)^L\|_{L^2}^2$ to both sides of
\eqref{G4.111}, together with the equivalence relation \eqref{G4.90} and frequency decomposition \eqref{G4.112}, yields the desired inequality \eqref{G4.91}.
\end{proof}

With inequality \eqref{G4.91} established, we proceed to derive the global time-decay estimate by introducing a supremum control functional. Define
\begin{align*}
 \mathcal M_1(t):=
 \sup_{0\leq\tau\leq t}
 (1+\tau)^{\frac{1-s}{2}}
 \|\nabla(\varrho,\omega)(\tau)\|_{L^2}.
\end{align*}
Substituting \eqref{G4.86}--\eqref{G4.88} into \eqref{G4.85}, we obtain
\begin{align}\label{G4.114}
 (1+t)^{\frac{1-s}{2}}
 \|\nabla(\varrho,\omega)^L(t)\|_{L^2}
 \leq C\bigl(1+\varepsilon\mathcal M_1(t)\bigr).
\end{align}
On the other hand, integrating \eqref{G4.91} over $(0,t)$ gives
\begin{align}\label{G4.115}
 \mathfrak E_1(t)
 \lesssim e^{-\lambda_4t}\mathfrak E_1(0)
 +\int_0^te^{-\lambda_4(t-\tau)}
 \|\nabla(\varrho,\omega)^L(\tau)\|_{L^2}^2\,\mathrm d\tau.
\end{align}
Since
\begin{align*}
 \int_0^te^{-\lambda_4(t-\tau)}(1+\tau)^{s-1}\,\mathrm d\tau
 \lesssim(1+t)^{s-1},
\end{align*}
the relations \eqref{G4.90}, \eqref{G4.114} and \eqref{G4.115} imply
\begin{align*}
 \mathcal M_1(t)
 \leq C\bigl(1+\varepsilon\mathcal M_1(t)\bigr).
\end{align*}
The choice of $\varepsilon_1$ yields the uniform bound of $\mathcal M_1(t)$, and thus
\begin{align}\label{G4.118}
 \|\nabla(\varrho,\omega)(t)\|_{L^2}
 \leq C(1+t)^{-\frac{1-s}{2}}.
\end{align}
Combining \eqref{G4.69} and \eqref{G4.118}, and recalling that
$\varrho=\rho-\rho_*$, we have
\begin{align}\label{G4.119}
 \|\nabla^k(\rho-\rho_*,\omega)(t)\|_{L^2}
 \leq C(1+t)^{-\frac{k-s}{2}},
 \qquad k=0,1.
\end{align}
Thus, the proof of Theorem~\ref{T1.2} is completed. \qed

\section{Lower bounds and optimality of the decay rates}\label{S5}

In this section, we prove Theorem~\ref{T1.4} by applying the
localized energy \eqref{G4.39} to the low-frequency blocks selected
by \eqref{G1.16}. After finishing Theorem~\ref{T1.4}, Corollary~\ref{cor1.5} will be verified as a direct consequence at the end of this section.
We keep the parameter $\eta_3$ fixed in
\eqref{G4.1a} and choose
\begin{align*}
 0<2\lambda_6<\lambda_5,
 \qquad
 \lambda_6<\frac{\lambda_4}{8}.
\end{align*}

\subsection{A localized energy inequality}

For $j\leq j_0$, let $\mathscr E_j$ and the signed quantity
$\mathscr Q_j$ be defined by \eqref{G4.39} and \eqref{G4.42},
respectively. Set
\begin{align*}
 \mathscr D_j:={}&
 2\mu_1\|\nabla\dot\Delta_j\omega\|_{L^2}^2
 +2\mu_2\|\dive\dot\Delta_j\omega\|_{L^2}^2
 +\frac{\eta_3\gamma}{1+2^{2j}}
 \|\nabla\dot\Delta_j\varrho\|_{L^2}^2\nonumber\\
 &-\frac{\eta_3\gamma}{1+2^{2j}}
 \|\dive\dot\Delta_j\omega\|_{L^2}^2
 -\frac{\eta_3\mu_1}{1+2^{2j}}
 \langle\nabla\dot\Delta_j\varrho,
 \Delta\dot\Delta_j\omega\rangle\nonumber\\
 &-\frac{\eta_3\mu_2}{1+2^{2j}}
 \langle\nabla\dot\Delta_j\varrho,
 \nabla\dive\dot\Delta_j\omega\rangle.
\end{align*}
Multiplying \eqref{G4.37} by $2$ and \eqref{G4.38} by
$\eta_3(1+2^{2j})^{-1}$ and adding the resulting identities gives
\begin{align*}
 \frac{\mathrm d}{\mathrm dt}\mathscr E_j(t)
 +\mathscr D_j(t)=\mathscr Q_j(t).
\end{align*}
By \eqref{G4.1a}, \eqref{G4.40}, and Young's inequality,
\begin{align*}
 \frac{\eta_3(\mu_1+\mu_2)}{1+2^{2j}}
 \big|\langle\nabla\dot\Delta_j\varrho,
 \nabla^2\dot\Delta_j\omega\rangle\big|\nonumber
 \leq\,
 \frac{\eta_3\gamma}{4(1+2^{2j})}
 \|\nabla\dot\Delta_j\varrho\|_{L^2}^2
 +C\eta_3\frac{2^{2j}}{1+2^{2j}}
 \|\nabla\dot\Delta_j\omega\|_{L^2}^2.
\end{align*}
Since
\begin{align*}
 \frac{2^{2j}}{1+2^{2j}}\mathscr E_j(t)
 \sim
 2^{2j}\|\dot\Delta_j(\varrho,\omega)(t)\|_{L^2}^2,
 \qquad j\leq j_0,
\end{align*}
we have
\begin{align}\label{G5.5}
 2\lambda_6 2^{2j}\mathscr E_j(t)
 \leq\mathscr D_j(t)
 \leq(\lambda_5-\lambda_6)2^{2j}\mathscr E_j(t),
 \qquad j\leq j_0.
\end{align}

By \eqref{G4.58}, set
\begin{align}\label{G5.12}
 \mathcal E_{\infty}:={}&
 \sup_{t\geq0}\|(\varrho,\omega)(t)\|_{\dot B^s_{2,\infty}}
 \leq C
 \|(\varrho_0,\omega_0)\|_{\dot B^s_{2,\infty}}.
\end{align}
Let $(b_j(t))_{j\in\mathbb Z}$ be the measurable frequency envelope
fixed in \eqref{G4.46}--\eqref{G4.46.2}, \eqref{G4.52} and \eqref{G4.53}--\eqref{G4.56}. Thus
\begin{align}\label{G5.8}
 b_j(t)\geq0,
 \qquad
 \sum_{j\in\mathbb Z}b_j(t)^2\leq1.
\end{align}
Using \eqref{G4.46}--\eqref{G4.56} together with the smallness margin $\lambda_6$
gives
\begin{align}\label{G5.9}
 2^{2sj}|\mathscr Q_j(t)|
 \leq
 \lambda_6 2^{2j}
 \big(2^{2sj}\mathscr E_j(t)\big)
 +C\varepsilon 2^{(1-s)j}\mathcal E_{\infty}^2
 +Cb^2_j(t)\mathcal D_3(t)\mathcal E_{\infty}^2.
\end{align}
Combining \eqref{G5.5} and \eqref{G5.9}, we obtain
\begin{align}\label{G5.11}
 \frac{\mathrm d}{\mathrm dt}
 \big(2^{2sj}\mathscr E_j(t)\big)
 +\lambda_5 2^{2j}
 \big(2^{2sj}\mathscr E_j(t)\big)\geq
 -C\varepsilon 2^{(1-s)j}\mathcal E_{\infty}^2
 -Cb^2_j(t)\mathcal D_3(t)\mathcal E_{\infty}^2.
\end{align}
For any $j\in\mathbb Z$, define
\begin{align*}
 r_j:={}&\int_0^\infty b^2_j(\tau)\mathcal D_3(\tau)\,\mathrm d\tau.
\end{align*}
Since $b_j(\tau)\to0$ as $j\to-\infty$ for almost every $\tau$, combined with the uniform bound \eqref{G3.60} and normalization \eqref{G5.8}, the dominated convergence theorem gives
\begin{align}\label{G5.14}
 \lim_{j\to-\infty}r_j=0.
\end{align}
Integrating \eqref{G5.11} over $(0,t)$ and using \eqref{G5.12} gives
\begin{align}\label{G5.15}
 2^{2sj}\mathscr E_j(t)
 \geq
 e^{-\lambda_5 2^{2j}t}
 2^{2sj}\mathscr E_j(0)
 -C\mathcal E_{\infty}^2
 \big(r_j+\varepsilon 2^{(-1-s)j}\big),
 \qquad j\leq j_0.
\end{align}

\subsection{Selection of the diffusive dyadic block}

By \eqref{G1.16}, the fixed relation between $u_0$ and $\omega_0$,
and the definition \eqref{G1.15}, there are constants
$\mathfrak a_*>0$, $L_*>0$, and a strictly decreasing sequence
$\{j_m\}_{m\geq1}$ as in \eqref{G1.15}, such that
\begin{align}\label{G5.16}
 j_m\rightarrow-\infty,
 \qquad
 1\leq j_m-j_{m+1}\leq L_*,\qquad
 2^{sj_m}
 \|\dot\Delta_{j_m}(\varrho_0,\omega_0)\|_{L^2}
 \geq\mathfrak a_*.
\end{align}
In view of \eqref{G4.40}, there is a constant $c>0$ such that
\begin{align}\label{G5.17}
 2^{2sj_m}\mathscr E_{j_m}(0)\geq
 c\mathfrak a_*^2.
\end{align}
For any sufficiently large $t$, choose $m=m(t)$ so that
\begin{align*}
 2^{2j_m}(1+t)\leq1
 <2^{2j_{m-1}}(1+t).
\end{align*}
The gap condition in \eqref{G5.16} implies
\begin{align}\label{G5.19}
 2^{-2L_*}(1+t)^{-1}
 <2^{2j_m}\leq(1+t)^{-1}.
\end{align}
Consequently,
\begin{align}\label{G5.20}
 e^{-\lambda_5 2^{2j_m}t}\geq e^{-\lambda_5}.
\end{align}
Moreover, $m(t)\to\infty$ as $t\to\infty$. By \eqref{G5.14}, there
exists $T_*>0$ such that, for $t\geq T_*$,
\begin{align}\label{G5.21}
 C\mathcal E_{\infty}^2
 \left(r_{j_{m(t)}}+\varepsilon
 2^{(-1-s)j_{m(t)}}\right)
 \leq\frac12e^{-\lambda_5}c\mathfrak a_*^2.
\end{align}
Substituting
\eqref{G5.17}, \eqref{G5.20} and \eqref{G5.21} into \eqref{G5.15}
yields
\begin{align}\label{G5.22}
 2^{2sj_{m(t)}}\mathscr E_{j_{m(t)}}(t)
 \geq\frac12e^{-\lambda_5}c\mathfrak a_*^2,
 \qquad t\geq T_*.
\end{align}

For $k=0,1$, the Littlewood--Paley almost-orthogonality,
Bernstein inequalities, \eqref{G4.40} and \eqref{G5.22} give
\begin{align*}
 \|\nabla^k(\varrho,\omega)(t)\|_{L^2}^2
 &\geq C^{-1}2^{2kj_{m(t)}}
 \|\dot\Delta_{j_{m(t)}}(\varrho,\omega)(t)\|_{L^2}^2\nonumber\\
 &\geq C^{-1}2^{2(k-s)j_{m(t)}}
 \big(2^{2sj_{m(t)}}\mathscr E_{j_{m(t)}}(t)\big)\nonumber\\
 &\geq C^{-1}\mathfrak a_*^2(1+t)^{-(k-s)},
 \qquad t\geq T_*,
\end{align*}
where \eqref{G5.19} was used in the last line. Therefore,
\begin{align}\label{G5.24}
 \|\nabla^k(\rho-\rho_*,\omega)(t)\|_{L^2}
 \geq c(1+t)^{-\frac{k-s}{2}},
 \qquad k=0,1,
 \qquad t\geq T_*.
\end{align}
The upper bound in \eqref{G1.17} is exactly \eqref{G4.119}.
Combining \eqref{G4.119} and \eqref{G5.24} yields
Theorem~\ref{T1.4}. \hfill $\Box$

Having established the sharp decay lower and upper bounds in Theorem~\ref{T1.4}, we now proceed to prove Corollary 1.5, which gives explicit optimal decay rates.
\subsection{Proof of Corollary~\ref{cor1.5}}

Assume first that
\begin{align*}
 m_\rho:=\int_{\mathbb R^3}(\rho_0-\rho_*)(x)\,\mathrm dx\neq0.
\end{align*}
Since $\rho_0-\rho_*\in L^1$, its Fourier transform is continuous and
\begin{align*}
 \widehat{\rho_0-\rho_*}(0)=m_\rho.
\end{align*}
Hence, there exists $r_*>0$ such that
\begin{align}\label{G5.27}
 |\widehat{\rho_0-\rho_*}(\xi)|
 \geq\frac{|m_\rho|}{2},
 \qquad |\xi|\leq r_*.
\end{align}
Let $\varsigma$ denote the frequency cutoff symbol associated with the dyadic projection $\dot\Delta_j$. There
exist an annulus $\mathcal A\Subset\{\frac34\leq|\xi|\leq\frac83\}$
and $c_\varsigma>0$ such that $|\varsigma(\xi)|\geq c_\varsigma$ on
$\mathcal A$. For all sufficiently negative $j$, Plancherel's
identity and \eqref{G5.27} give
\begin{align*}
 \|\dot\Delta_j(\rho_0-\rho_*)\|_{L^2}^2
 &\geq
 \int_{2^j\mathcal A}
 |\varsigma(2^{-j}\xi)|^2
 |\widehat{\rho_0-\rho_*}(\xi)|^2\,\mathrm d\xi\geq c_\varsigma^2\frac{|m_\rho|^2}{4}
 |\mathcal A|2^{3j}.
\end{align*}
Thus
\begin{align*}
 2^{-\frac32j}
 \|\dot\Delta_j(\rho_0-\rho_*)\|_{L^2}
 \geq c|m_\rho|,
\end{align*}

Similarly, if
\begin{align*}
 m_u:=\int_{\mathbb R^3}u_0(x)\,\mathrm dx\neq0,
\end{align*}
then applying the same calculation componentwise to $u_0$ and using
$\omega_0=\rho_\infty[P'(\rho_\infty)]^{-1/2}u_0$ yields
\begin{align*}
 2^{-\frac32j}\|\dot\Delta_j\omega_0\|_{L^2}
 \geq c|m_u|
\end{align*}
for all sufficiently negative $j$. Hence, the same conclusion holds.
Taking $s=-\frac32$ in Theorem~\ref{T1.4} proves
\begin{align*}
 c(1+t)^{-\frac{2k+3}{4}}
 \leq
 \|\nabla^k(\rho-\rho_*,\omega)(t)\|_{L^2}
 \leq
 C(1+t)^{-\frac{2k+3}{4}},
 \qquad k=0,1,
\end{align*}
for all sufficiently large $t$, which implies
Corollary~\ref{cor1.5}. \hfill $\Box$

\bigskip
\textbf{Acknowledgments.}
J. Ni would like to express his sincere gratitude to Prof. Renjun Duan for his valuable discussions and guidance. The research of L. Wang was partially supported by the Basic Research Program of Jiangsu Province (Grant No. BK20240058), the China Postdoctoral Science Foundation
(Grant No. 2024M751365), and the Jiangsu Funding Program for Excellent Postdoctoral Talent (Grant No. 2023ZB071). The research of Z. Zhang was supported by the National Natural Science Foundation of China (Grant Nos. 12471215 and 12331007) and the Taishan Scholars Foundation of Shandong Province (Grant No. tsqn202507101).

\vspace{2mm}

\textbf{Conflict of interest.} The authors declare that they have no conflicts of interest.

\vspace{2mm}

\textbf{Data availability statement.}
 No datasets were generated or analyzed during the current study.

\raggedbottom
\bibliographystyle{plain}

\end{document}